\documentclass[12pt]{amsproc}
\usepackage{amssymb}
\usepackage{amsxtra}
\usepackage{graphics}
\usepackage{diagrams}
\usepackage{latexsym}
\usepackage{enumerate}
\usepackage{xcolor}
\usepackage{amsmath}
\usepackage{amssymb,amsthm,amsfonts}
\usepackage{mathtools}
\usepackage{amscd}
\usepackage[arrow, matrix, curve]{xy}
\usepackage[mathscr]{euscript}
\usepackage{syntonly}
\usepackage{adjustbox}

\begin{document}

%%%%%%%%%%%%%%%%%%%% Text italic %%%%%%%%%%%%%%%%%%%%%%%%%%%%
\theoremstyle{plain}
\newtheorem{thm}{Theorem}[section]
\newtheorem{theorem}[thm]{Theorem}
\newtheorem{lemma}[thm]{Lemma}
\newtheorem{corollary}[thm]{Corollary}
\newtheorem{proposition}[thm]{Proposition}
\newtheorem{addendum}[thm]{Addendum}
\newtheorem{variant}[thm]{Variant}
\newtheorem*{thm0}{Theorem}
\newtheorem*{thm35}{Theorem 3.5}
%%%%%%%%%%%%%%%%%%%% Text roman %%%%%%%%%%%%%%%%%%%%%%%%%%%%%
\theoremstyle{definition}
\newtheorem{construction}[thm]{Construction}
\newtheorem{notations}[thm]{Notations}
\newtheorem{question}[thm]{Question}
\newtheorem{problem}[thm]{Problem}
\newtheorem{remark}[thm]{Remark}
\newtheorem{remarks}[thm]{Remarks}
\newtheorem{definition}[thm]{Definition}
\newtheorem{claim}[thm]{Claim}
\newtheorem{assumption}[thm]{Assumption}
\newtheorem{assumptions}[thm]{Assumptions}
\newtheorem{properties}[thm]{Properties}
\newtheorem{example}[thm]{Example}
\newtheorem{conjecture}[thm]{Conjecture}
\newtheorem{sublemma}[thm]{Sublemma}
\numberwithin{equation}{thm}
% Skriptbuchstaben
\newcommand{\pP}{{\mathfrak p}}
\newcommand{\sA}{{\mathcal A}}
\newcommand{\sB}{{\mathcal B}}
\newcommand{\sC}{{C}}
\newcommand{\sD}{{\mathcal D}}
\newcommand{\sE}{{\mathcal E}}
\newcommand{\sF}{{\mathcal F}}
\newcommand{\sG}{{\mathcal G}}
\newcommand{\sH}{{\mathcal H}}
\newcommand{\sI}{{\mathcal I}}
\newcommand{\sJ}{{\mathcal J}}
\newcommand{\sK}{{\mathcal K}}
\newcommand{\sL}{{\mathcal L}}
\newcommand{\sM}{{\mathcal M}}
\newcommand{\sN}{{\mathcal N}}
\newcommand{\sO}{{\mathcal O}}
\newcommand{\sP}{{\mathcal P}}
\newcommand{\sQ}{{\mathcal Q}}
\newcommand{\sR}{{\mathcal R}}
\newcommand{\sS}{{\mathcal S}}
\newcommand{\sT}{{\mathcal T}}
\newcommand{\sU}{{\mathcal U}}
\newcommand{\sV}{{\mathcal V}}
\newcommand{\sW}{{\mathcal W}}
\newcommand{\sX}{{\mathcal X}}
\newcommand{\sY}{{\mathcal Y}}
\newcommand{\sZ}{{\mathcal Z}}
% Sonderbuchstaben mit Doppellinie
\newcommand{\A}{{\mathbb A}}
\newcommand{\B}{{\mathbb B}}
\newcommand{\C}{{\mathbb C}}
\newcommand{\D}{{\mathbb D}}
\newcommand{\E}{{\mathbb E}}
\newcommand{\F}{{\mathbb F}}
\newcommand{\G}{{\mathbb G}}
\renewcommand{\H}{{\mathbb H}}
\newcommand{\I}{{\mathbb I}}
\newcommand{\J}{{\mathbb J}}
\renewcommand{\L}{{\mathbb L}}
\newcommand{\M}{{\mathbb M}}
\newcommand{\N}{{\mathbb N}}
\renewcommand{\P}{{\mathbb P}}
\newcommand{\Q}{{\mathbb Q}}
\newcommand{\R}{{\mathbb R}}
\newcommand{\SSS}{{\mathbb S}}
\newcommand{\T}{{\mathbb T}}
\newcommand{\U}{{\mathbb U}}
\newcommand{\V}{{\mathbb V}}
\newcommand{\W}{{\mathbb W}}
\newcommand{\Z}{{\mathbb Z}}
\newcommand{\g}{{\gamma}}
\newcommand{\bb}{{\beta}}
\newcommand{\as}{{\alpha}}
\newcommand{\id}{{\rm id}}
\newcommand{\rk}{{\rm rank}}
\newcommand{\END}{{\mathbb E}{\rm nd}}
\newcommand{\End}{{\rm End}}
\newcommand{\Hom}{{\rm Hom}}
\newcommand{\Hg}{{\rm Hg}}
\newcommand{\tr}{{\rm tr}}
\newcommand{\Sl}{{\rm Sl}}
\newcommand{\Gl}{{\rm Gl}}
\newcommand{\Cor}{{\rm Cor}}
\newcommand{\HIG}{\mathrm{HIG}}
\newcommand{\MIC}{\mathrm{MIC}}
\newcommand{\Res}{\textrm{Res}}
\newcommand{\Pic}{\textrm{Pic}}
\newcommand{\Spec}{\textrm{Spec}}
\newcommand{\SO}{{\rm SO}}
\newcommand{\OO}{{\rm O}}
\newcommand{\SP}{{\rm SP}}
\newcommand{\Sp}{{\rm Sp}}
\newcommand{\UU}{{\rm U}}
\newcommand{\SU}{{\rm SU}}
\newcommand{\SL}{{\rm SL}}
\newcommand{\tL}{{\mathscr L}}
\newcommand{\Ext}{{\rm{Ext}}}

%%%%%%%%%%%%%%%%%%%%%%%%%%%%%%%%%%%%%%%%%%%%%%%%%%%%%%%%

\title[]{Uniformization of $p$-adic curves via Higgs-de Rham flows}

\author[G.-T. Lan]{Guitang Lan}

\author[M. Sheng]{Mao Sheng}
\email{msheng@ustc.edu.cn}
\address{School of Mathematical Sciences,
University of Science and Technology of China, Hefei, 230026, China}
\author[Y.-H. Yang]{Yanhong Yang}
\author[K. Zuo]{Kang Zuo}
\email{lan@uni-mainz.de, yangy@uni-mainz.de, zuok@uni-mainz.de}
\address{Institut f\"{u}r  Mathematik, Universit\"{a}t
Mainz, Mainz, 55099, Germany}
\thanks{This work is supported by the SFB/TR 45 ``Periods, Moduli Spaces and Arithmetic of Algebraic Varieties" of the DFG. The second named author is supported by National Natural Science Foundation of China (Grant No. 11471298, No. 11526212.)}

\begin{abstract}
Let $k$ be an algebraic closure of a finite field of odd characteristic. We prove that for any rank two graded Higgs bundle with maximal Higgs field over a generic hyperbolic curve $X_1$ defined over $k$, there exists a lifting $X$ of the curve to the ring $W(k)$ of Witt vectors as well as a lifting of the Higgs bundle to a periodic Higgs bundle over $X/W(k)$. These liftings  give rise to a two-dimensional absolutely irreducible representation of the arithmetic fundamental group $\pi_1(X_K)$ of the generic fiber of $X$. This curve $X$ and its associated representation is in close relation to the canonical curve and its associated canonical crystalline representation in the $p$-adic Teichm\"{u}ller theory for curves due to S. Mochizuki. Our result may be viewed as an analogue of the Hitchin-Simpson's uniformization theory of hyperbolic Riemann surfaces via Higgs bundles.
\end{abstract}

\maketitle

\tableofcontents
\section{Introduction}

The work of Deninger-Werner \cite{DW} initiated the problem of associating representations of geometric fundamental groups of $p$-adic curves to vector bundles with suitable conditions. Their result can be viewed as a partial analogue of the classical Narasimhan-Seshadri theory of vector bundles on
compact Riemann surfaces. Later, Faltings \cite{Fa}, using the theory of almost \'{e}tale extensions in his work in the $p$-adic Hodge theory, obtained a far-reaching generalization. In particular, as a $p$-adic analogue of Simpson's theory in the nonabelian Hodge theory, Faltings associates generalized representations to Higgs bundles over $p$-adic curves. However, a fundamental question concerning semistability of Higgs bundles remains: Faltings asked whether semistable Higgs bundles of degree zero over a smooth projective $p$-adic curve corresponds to usual (i.e. continuous $\C_p$-) representations of the geometric fundamental group.

Reading the foundational paper \cite{Hi} of Hitchin in the nonabelian Hodge theory, one finds, on the other hand, that even the question that whether a very basic graded Higgs bundle of the following type corresponds a usual representation is unknown:
\begin{align}\label{Higgsbd0}
&(E,\theta)\ :=\ (L\oplus L^{-1},\theta), \\
&\text{
where }L \text{
is a line bundle over } X \text{ satisfying  }L^2\cong \Omega^1_{X}, \nonumber \\
&\text{
and }\theta: L \to L^{-1}\otimes\Omega^1_{X}\text{ is
the tautological  isomorphism.}\nonumber
 \end{align}

For $X$ being a compact Riemann surface, this is Example 1.5 \cite{Hi} (Example 1.4 relates to the theory of Narasimhan-Seshadri on stable vector bundles). Such a Higgs bundle is considered to be basic because, as shown by Hitchin, it gives a uniformization for a hyperbolic $X$. Simpson \cite{Si} extended the theory to a noncompact Riemann surface by introducing the logarithmic version of the previous Higgs bundles, namely
\begin{align}\label{Higgsbdl}
&(E,\theta)\ :=\ (L\oplus L^{-1},\theta), \\
&\text{
where }L \text{
is a line bundle over } X \text{ satisfying  }L^2\cong \Omega^1_{X}(\log
D), \nonumber \\
&\text{
and }\theta: L \to L^{-1}\otimes\Omega^1_{X}(\log D) \text{ is
a  tautological  isomorphism,}\nonumber
\end{align}
where $X$ is smooth projective and $D\subset X$ a reduced divisor. This paper is devoted to study this type of graded (logarithmic) Higgs bundles in the context of the $p$-adic Hitchin-Simpson correspondence established in \cite{LSZ}. For a rank two graded logarithmic Higgs bundle with trivial determinant, the above type is characterized by the Higgs field being an isomorphism and it is said to have \emph{maximal Higgs field}. Note that examples of the above type are typically Higgs stable and are of degree zero. Recall that in the complex case, a Higgs bundle with maximal Higgs field \eqref{Higgsbd0} is used to recover the uniformization theorem of a compact Riemann surface by solving the Yang-Mills-Higgs equation. Hitchin observed that the unique solution to
the Yang-Mills-Higgs equation associates to $(E,\theta)$  defined in \eqref{Higgsbd0} over $X$  a
polarized $\mathbb C$-variation of Hodge structure $(H,\nabla,
Fil , \Psi)$ (in fact it carries also a real structure), where $H$ is the underlying
$C^\infty$-bundle of $E$ with a new holomorphic
structure, $\nabla$ is an integrable connection
$$\nabla: H\to H\otimes\Omega^1_{X},$$
$Fil$ is a Hodge filtration, that is, a finite decreasing filtration satisfying Griffiths transversality and $\Psi$ is a horizontal
bilinear form satisfying the Hodge-Riemann bilinear relation. By taking the
grading of the Hodge filtration, one obtains that the associated
graded Higgs bundle $Gr_{Fil }(H,\nabla)$ is isomorphic to  $(E,\theta)$.
Moreover, the classifying map associated to $(H,\nabla,
Fil, \Psi)$ is a holomorphic map
$$\pi: \tilde X\to \mathcal H $$ from the universal cover $\tilde X$ of $X$ to the classifying space of rank two polarized
$\mathbb R$-Hodge structures of weight one, which turns out to be the
upper half plane $\mathcal H$. As the derivative of $\pi$ can be
identified with $\theta$ via the grading
$Gr_{Fil}(H,\nabla)$,  one knows that $\pi$ is an
isomorphism. This illustrates the approach to the uniformization
theorem from the point of view of Higgs bundles.

Turn to the $p$-adic case. Let $k:= \bar{\mathbb{F}}_p$ be an algebraic closure of a finite field of odd characteristic $p$,  $W:=W(k)$ the ring of Witt vectors with coefficients in $k$, and $K$ its fraction field, $\bar{K}$ an algebraic closure of $K$. Based on the fundamental work of Ogus-Vologodsky \cite{OV} on nonabelian Hodge theory in characteristic  $p$, we introduced in \cite{LSZ} the notion of a strongly semistable Higgs bundle, generalizing the notion of a strongly semistable vector bundle which played an important role in the work \cite{DW}, and a characteristic  $p$/$p$-adic analogue of Yang-Mills-Higgs flow whose ``limit" can be regarded as a characteristic $p$/$p$-adic analogue of the solution of the Yang-Mills-Higgs equation, whose definition is recalled as follows.

Let $X_1$ be a smooth projective variety over $k$ together with a simple normal crossing divisor $D_1\subset X_1$ such that $(X_1,D_1)$ is $W_2(k)$-liftable. Fix a smooth projective scheme $X_2$ over $W_2:=W_2(k)$ together with a divisor $D_2\subset X_2$ relative to $W_2(k)$ whose reduction mod $p$ is $(X_1,D_1)$. Let $C_1^{-1}$ be the inverse Cartier transform of Ogus-Vologodsky from the category of nilpotent logarithmic Higgs module of exponent $\leq p-1$ to the category of nilpotent logarithmic flat module of exponent $\leq p-1$ with respect to the chosen $W_2$-lifting $(X_2,D_2)$. We refer the reader to \cite{LSZ0} for an elementary approach to the construction of the inverse Cartier/Cartier transform in the case $D_1=\emptyset$ and the Appendix in a special log case. A Higgs-de Rham flow over $X_1$ is a diagram of the following form:
$$
\xymatrix{
                &  (H_0,\nabla_0)\ar[dr]^{Gr_{Fil_0}}       &&  (H_1,\nabla_1)\ar[dr]^{Gr_{Fil_1}}    \\
 (E_0,\theta_0) \ar[ur]^{C_1^{-1}}  & &     (E_1,\theta_1) \ar[ur]^{C_1^{-1}}&&\ldots       }
$$
where the initial term $(E_0,\theta_0)$ is a nilpotent graded Higgs bundle with exponent $\leq p-1$; for $ i\geq 0$, $Fil_i$ is a Hodge filtration on the flat bundle $C_1^{-1}(E_i,\theta_i)$ of level $\leq p-1$; for $i\geq 1$,  $(E_i,\theta_i)$ is the graded Higgs bundle associated with the de Rham bundle $(C_1^{-1}(E_{i-1},\theta_{i-1}),Fil_{i-1})$. A Higgs-de Rham flow is said to be \emph{periodic} of period $f\in \N$ if $f$ is the minimal integer such that there exists an isomorphism of graded Higgs modules
$\phi: (E_f,\theta_f)\cong (E_0,\theta_0)$. A (logarithmic) Higgs bundle is said to be periodic if it initiates a periodic Higgs-de Rham flow.
There are several points we want to emphasize in the above definition of a (periodic) flow: first, the Hodge filtrations in the above diagram do not come from the inverse Cartier transform, but rather are a part of the defining data  of a flow; second, the choice of the isomorphism $\phi$ is also a part of the defining data of a periodic flow; third, since the inverse Cartier transform does depend on the choice of a $W_2$-lifting of $(X_1,D_1)$, a periodic Higgs bundle only makes sense after a $W_2$-lifting is specified, although the Higgs bundle itself is just defined over $k$. To define the notion of a periodic Higgs-de Rham flow over a truncated Witt ring, one needs to lift the inverse Cartier transform of Ogus-Vologodsky which has been partially realized in \cite{LSZ}. We refer the reader to Section 5 in that work for details. It is quite straightforward to generalize one of main results \cite[Theorem 1.6]{LSZ} to the following logarithmic case, which shall provide us with the basic device for producing representations in this paper.
\begin{theorem}
Let $Y$ be a smooth projective scheme over $W$ with a simple normal crossing divisor $D\subset Y$ relative to $W$. Then for each natural number $f\in \N$, there is an equivalence of categories between the category of strict $p^n$-torsion logarithmic Fontaine modules (with pole along $D\times W_n\subset Y\times W_n$) with endomorphism structure of $W_n(\F_{p^f})$ whose Hodge-Tate weight $\leq p-2$ and the category of periodic Higgs-de Rham flows over $Y\times W_n$ whose periods are divisors of $f$ and exponents of nilpotency are $\leq p-2$, where $W_n(\F_{p^f})$ (resp. $W_n$) is the truncated Witt ring $W(\F_{p^f})/p^n$ (resp. $W/p^n$) with coefficients in $\F_{p^f}$ (resp. $k$).
\end{theorem}
After the fundamental theorem of Faltings \cite[Thereom 2.6*, Page 43 i)]{Fa0}, to each periodic Higgs-de Rham flow over $Y_n$ whose period is a divisor  of $f$ as above, one can now associate a crystalline  representation of the arithmetic fundamental group of the generic fiber $Y^\circ_K:=(Y-D)\times K$ with coefficients in $W_n(\F_{p^f})$. According to the theorem, the problem is reduced to showing that \eqref{Higgsbdl} over $k$ is periodic and can be lifted to a periodic Higgs bundle over an arbitrary truncated Witt ring $W_n, n\in \N$. Now we come to our main theorem.
\begin{thm}\label{mainthm}
Assume that $2g-2+r>0$ and $r$ is even. Let $X_1$ be a generic curve in the moduli  space $\mathscr M_{g,r}$ of smooth projective curves over $k$ with $r$ marked points $D$. Let $(L_1\oplus L_1^{-1},\theta_1)$ be a logarithmic Higgs bundle with maximal Higgs field \eqref{Higgsbdl} defined over $X_1$.
Then there exists a tower of log smooth  liftings $(X_n, D_n)/ W_n, n\in \N$
\begin{align}\label{lifts}
(X_1,D_1:=D)\hookrightarrow  (X_2,D_2)\hookrightarrow \cdots \hookrightarrow  (X_n,D_n)\hookrightarrow  \cdots,
\end{align}
such that $(L_1\oplus L_1^{-1},\theta_1)$ is two-periodic with respect to the $W_2$-lifting $(X_2,D_2)$ and lifts to a two-periodic logarithmic Higgs bundle $(L_n\oplus L_n^{-1},\theta_n)$ over $X_n$ with the log pole along $D_n$ for each $n\geq 2$.

To be more specific,  for all $n\in \N$, there exists a log smooth $W_{n+1}$-lifting
  $(X_{n+1},D_{n+1})$ of $(X_{n},D_n)$, a logarithmic Higgs bundle $(L_n\oplus L_n^{-1},\theta_n)$ over $X_n$, a two-torsion line bundle $\mathscr{L}_{n}$ over $X_{n}$,  a Hodge filtration $Fil_{n} $ on the inverse Cartier transform $C^{-1}_{n}(L_{n}\oplus
  L_{n}^{-1},\theta_{n})$ with respect to  $X_{n}\subset X_{n+1}$, and an isomorphism
  $$\phi_{n}: Gr_{Fil_{n} }\circ C_{n}^{-1}(L_{n}\oplus
  L_{n}^{-1},\theta_{n}) \cong (L_{n}\oplus
  L_{n}^{-1},\theta_{n})\otimes (\mathscr{L}_{n},0) ,$$
 such that for all $n\geq 2$,
  \begin{align*}
  &(L_{n}\oplus L_{n}^{-1},\theta_n) \equiv (L_{n-1}\oplus L_{n-1}^{-1},\theta_{n-1})  \text{ mod }p^{n-1};
  \\
  &\mathscr{L}_{n}  \equiv \mathscr{L}_{n-1} \text{ mod }p^{n-1}; \quad  Fil_{n} \equiv Fil_{n-1}  \text{ mod }p^{n-1}; \quad \phi_{n}  \equiv \phi_{n-1} \text{ mod
  }p^{n-1}.
\end{align*}

Set $(E_n, \theta_n):= (L_n\oplus
L_n^{-1},\theta_n)$, and denote the trivial filtration by $Fil_{tr}$. Then there is a tower of two-periodic
Higgs-de Rham flows as below:
\begin{align} \label{tower}
\begin{adjustbox}{scale=0.8}
\adjustbox{scale=0.75}{
\xymatrix{
 &  ((H_{n+1},
\nabla_{n+1}),  Fil_{n+1})  \ar[rd]^{\text{Gr}}  \ar@{.>}[ddd]^(.6){\text{ mod }p^{n}}& & ((H_{n+1},
\nabla_{n+1})\otimes \mathscr{L}_{n+1},  Fil_{n+1}\otimes Fil_{tr})\ar[rd]^{\text{Gr}}   \ar@{.>}[ddd]^(.6){\text{ mod }p^{n}} &\\
(E_{n+1},\theta_{n+1})  \ar[ru]^{C^{-1}_{n+1}}   \ar@{.>}[ddd]^-{\text{ mod }p^{n}} & &   (E_{n+1},\theta_{n+1})\otimes \mathscr{L}_{n+1} \ar[ru]^{C^{-1}_{n+1}}    \ar@{.>}[ddd]^-{\text{ mod }  p^{n}}& &(E_{n+1},\theta_{n+1})    \ar@{.>}[ddd]^-{\text{ mod }p^{n}} \ar@/^1pc/[llll]|(.45){\cong}\\
&&&&\\
 & ( (H_n,
\nabla_n),Fil_n) \ar[rd]^{\text{Gr}} & &  ((H_n,
\nabla_n)\otimes \mathscr{L}_{n},  Fil_n\otimes Fil_{tr}) \ar[rd]^{\text{Gr}} & \\
(E_n, \theta_n) \ar[ru]^{C^{-1}_{n}} & & (E_n, \theta_n) \otimes \mathscr{L}_{n} \ar[ru]^{C^{-1}_{n}}& & (E_n, \theta_n)\ar@/^1pc/[llll]|(.45){\cong}
}}
\end{adjustbox}
\end{align}
\end{thm}
The following consequence is immediate from Theorems 1.1 and 1.2 .
\begin{corollary}\label{corollary}
Use notation as above. Then for a logarithmic Higgs bundle with maximal Higgs field over a generic curve $(X_1,D_1)$ in the moduli space $\mathscr M_{g,r}$, one has a log smooth curve  $(X_{\infty},D_{\infty})$ over $W$ lifting $(X_1,D_1)$ together with a two-dimensional irreducible crystalline  representation
\begin{align}
\rho: \pi_1(X_K) \to \text{GL}(2, W(\mathbb{F}_{p^2})),
\end{align}
where $X_K:=X\times_W K$ is the generic fiber of the hyperbolic curve $X:=X_{\infty}-D_{\infty}$ over $W$.
\end{corollary}
The so-constructed representation $\rho$ in the above corollary shares the following stronger irreducibility property. For simplicity, we give  the statement  only for the case where the divisor $D_1$ in $X_1$ is empty.
\begin{proposition}
Use notation as above. Assume $D_1=\emptyset$. Denote by $\bar{\rho}$ the restriction of $\rho$ to the geometric fundamental group  $\pi_1 (X_{\bar{K}})$. Then for any smooth curve $Y_{\bar{K}}$ over $\bar{K}$ with a finite morphism $f: Y_{\bar{K}} \to X_{\bar{K}}$, the induced representation of $\pi_1(Y_{\bar{K}})$ from $\bar \rho$ is absolutely irreducible.
\end{proposition}

\begin{proof} Let $(E,\theta)$ be the inverse limit of the graded Higgs bundle $\{(E_n,\theta_n)\}$ in Theorem \ref{mainthm}.
By the example in \cite[Page 861]{Fa}, one sees that the generalized representation corresponding to $(E,\theta)_{\C_p}:=(E,\theta)\otimes \C_p$ is compatible with $\bar{\rho}$, that is, it is just the scalar extension of $\bar{\rho}$ by tensoring with $\C_p$.  We can find a finite extension field  $K'$ of  $K$, with its integral ring $\mathcal{O}_{K'}$, such that $Y_{\bar K}$ has an integral model $Y_{\mathcal{O}_{K'}}$ over $\mathcal{O}_{K'}$ with toroidal singularity. By the construction of the correspondence \cite[Theorem 6]{Fa}, the twisted pullback of the graded Higgs bundle $f^{\circ}(E,\theta)_{\C_p}$ corresponds to the pull-back representation of $\bar{\rho}\otimes \C_p$ to $\pi_1 (Y_{\bar{K}})$. By the very construction of the twisted pullback, one has a short exact sequence
 $$
 0\to (f^*L^{-1},0)_{\C_p}\to f^{\circ}(E,\theta)_{\C_p}\to (f^*L,0)_{\C_p}\to 0,
 $$
and that the Higgs field of $f^{\circ}(E,\theta)_{\C_p}$ has nonzero Higgs field. Assume the contrary that the restricted $\C_p$-representation of $\pi_1 (Y_{\bar{K}})$ is not irreducible. Then it contains a one-dimensional $\C_p$-subrepresentation, and by the last paragraph in \cite[Page 860]{Fa}, it follows that $f^{\circ}(E,\theta)_{\C_p}$ contains a rank one Higgs subbundle $(N,0)$ of degree zero. But, this leads to a contradiction: as $\deg f^*L^{-1}<0$, the composite
$$
N\to   f^{\circ}(E,\theta)_{\C_p}\to f^*L_{\C_p}
$$
cannot be zero and hence an isomorphism over a nonempty open subset $U$. This implies that the Higgs field of $f^{\circ}(E,\theta)_{\C_p}$ is zero over $U$ and hence zero over the whole space which is impossible.
\end{proof}

In general, the $p$-adic curve appeared in Corollary \ref{corollary} is neither unique nor far from an arbitrary lifting of $(X_1,D_1)$ over $k$. In fact, it is in close relation with the notion of a \emph{canonical curve} due to S. Mochizuki \cite[Definition 3.1, Ch. III]{Mo}. At this point, it is necessary to make a brief clarification about the relation between the notion of an \emph{indigenous bundle}, which is central in the $p$-adic Teichm\"{u}ller theory of Mochizuki for curves, and the notion of a logarithmic Higgs bundle with maximal Higgs field (i.e. \eqref{Higgsbdl}). According to Mochizuki, an indigenous bundle is a $\P^1$-bundle with connection associated to a rank two flat bundle with a Hodge filtration, and its associated graded Higgs bundle is of the form
\eqref{Higgsbdl}.  He started with indigenous bundles over a Riemann surface admitting an integral structure over $W(k)$ for some $p$, and then
studied the moduli space of $p$-adic indigenous bundles over
the corresponding $p$-adic curve. Although related, the approach and setting in \cite{Mo} are very different from ours. From the point of view of nonabelian Hodge theory, one is by no means restricted to the curve case. As an illustration, our approach yields a similar result for an ordinary abelian variety $A$ as Theorem \ref{mainthm} by considering the following Higgs bundle (see \cite[Example 5.27]{LSZ}):
$$(\Omega_{A}\oplus\mathcal{O}_{A}, \theta), \text{ where }
\theta:  \Omega_{A}\to \mathcal{O}_{A}\otimes\Omega_{A} \text{ is the tautological isomorphism}.$$

Finally, we shall remind the reader that the explicit form of the divisor $D_1$ in the log curve $X_1$ plays a minor role in our paper and therefore will be suppressed whenever the context is clear. Thus the notation $X_1$ could sometimes actually mean the whole pair $(X_1,D_1)$. The same convention applies also for log curves $X_n$ over $W_n, n\geq 2$.

\section{Outline of the proof of the main theorem}

Let $(E_1:=L_1\oplus L_1^{-1},\theta_1)$ be the Higgs bundle
 \eqref{Higgsbdl}. The proof of Theorem
\ref{mainthm} is divided into two steps. In the first step, we show
that there exists a $W_{2}$-lifting of $X_1$ such that
\eqref{Higgsbdl} becomes a two-periodic Higgs bundle, see Theorem
\ref{thm03}; in the second step, we show that under some conditions,
a periodic Higgs bundle over $X_n$ is liftable to a periodic Higgs
bundle over $X_{n+1}$, see Theorem \ref{thm04}.

\begin{theorem}\label{thm03}
For a generic curve $X_1\in \mathscr{M}_{g,r}$,
there exists a  $W_2$-lifting $X_2$ of  $X_1$ and a Hodge filtration
$Fil_1$ on $C_{1}^{-1}(L_1\oplus L_1^{-1},\theta_{1})$
such that
\begin{align}\label{eq00}
Gr_{Fil_1}\circ C_{1}^{-1}(L_1\oplus
L_1^{-1},\theta_{1})\ \cong \ (L_1  \oplus L_1^{-1},\theta_1)\otimes
(\mathscr{L},0 ),
\end{align}
where $\mathscr{L}$ is a two-torsion line bundle over $X_1$.
\end{theorem}

\begin{corollary}\label{cor of thm 03}
Use notation as above. The Higgs bundle $(L_1\oplus
L_1^{-1},\theta_1)$ with respect to the $W_2$-lifting $X_2$ is two-periodic.
\end{corollary}
\begin{proof}
It suffices to show there is some filtration $Fil_2$ on
$$
C_1^{-1}((L_1  \oplus L_1^{-1},\theta_1)\otimes
(\mathscr{L},0))
$$
such that two graded Higgs bundles are isomorphic:
$$
Gr_{Fil_2}\circ C_{1}^{-1}\circ Gr_{Fil_1}\circ C_{1}^{-1}(L_1\oplus
L_1^{-1},\theta_{1})\ \cong \ (L_1  \oplus L_1^{-1},\theta_1).
$$
Since
$
C_1^{-1}((L_1  \oplus L_1^{-1},\theta_1)\otimes
(\mathscr{L},0))$ is naturally isomorphic to
$$
C_1^{-1}(L_1  \oplus L_1^{-1},\theta_1)\otimes
C_1^{-1}(\mathscr{L},0)=C_1^{-1}(L_1  \oplus L_1^{-1},\theta_1)\otimes
(\mathscr{L},\nabla_{can}),
$$
we can equip it with the filtration $Fil_2$ which is the tensor product of the filtration $Fil_1$ on $C_1^{-1}(L_1  \oplus L_1^{-1},\theta_1)$ and the trivial filtration on $\mathscr{L}$. In the previous equality, $\nabla_{can}$ denotes for the canonical connection in the Cartier descent theorem. Then by Theorem \ref{thm03}, it follows that
$$
Gr_{Fil_2}\circ C_{1}^{-1}\circ Gr_{Fil_1}\circ C_{1}^{-1}(L_1\oplus
L_1^{-1},\theta_{1})\cong (L_1\oplus
L_1^{-1},\theta_{1})\otimes (\mathscr{L}^{\otimes 2},0)=(L_1\oplus
L_1^{-1},\theta_{1}).
$$
\end{proof}
To outline the proof of Theorem \ref{thm03}, we start with an observation: For every $W_2$-lifting $X_2$ of  $X_1$, there exists a short exact sequence of flat bundles as follows:
\begin{align}\label{eq01}
0\to (F^*L_1^{-1}, \nabla_{can})\to C^{-1}_{X_1\subset
X_2}(E_1,\theta_1) \to (F^*L_1,\nabla_{can} )\to 0,
\end{align}
where $C^{-1}_{X_1\subset
X_2}$ is the inverse Cartier transform $C_1^{-1}$ and $F: X_1 \to X_1$ is the absolute Frobenius.  Here and also in the following, when the dependence of the inverse Cartier transform on the choice of a $W_2$-lifting $X_2$ of $X_1$ plays a prominent role in the situation, we shall use instead the notation $C^{-1}_{X_1\subset
X_2}$ for the inverse Cartier transform. The above short exact sequence comes after applying the exact functor $C^{-1}_{X_1\subset X_2}$ to the short exact sequence of Higgs bundles:
$$
0\to (L_1^{-1}, 0)\to (E_1,\theta_1) \to (L_1,0)\to 0.
$$
Set $(H,\nabla)=C^{-1}_{X_1\subset
X_2}(E_1,\theta_1)$. Then, ignoring the connection in the exact sequence, we get an extension of vector bundles:
\begin{align}\label{eq02}
0\to F^*L_1^{-1} \to H \to F^*L_1\to 0.
\end{align}
The cohomology group $H^1(X_1,F^*L_1^{-2})\cong \textrm{Ext}^1(F^*L_1,F^*L_1^{-1})$ classifies the isomorphism classes of extensions of form (\ref{eq02}). Let $\{X_1\subset X_2 \}$ be the set of  isomorphism classes of $W_2$-liftings of $X_1$, which is known to be an $H^1(X_1,\mathcal{T}_{X_1/k})$-torsor, where $\mathcal{T}_{X_1/k}$ is the log tangent sheaf of $X_1/k$. To consider the effect of the choices of $W_2$-liftings on the inverse Cartier transform of $(E_1,\theta_1)$, we shall consider the natural map
\begin{align}\label{eq03}
\rho: \{X_1\subset X_2 \} \to H^1(X_1,F^*L_1^{-2}),
\end{align}
obtained by sending $X_1\subset X_2$ to the extension class of $H$ as above. It turns out that the map $\rho$ behaves well under the torsor action and is injective (Lemma \ref{observs} and Lemma \ref{injective}). Set $A=\text{Im}(\rho)$. In order to take into account the connection in the inverse Cartier transform, we shall also consider the isomorphism classes $\textrm{Ext}^1\left((F^*L_1,\nabla_{can}),(F^*L_1^{-1},\nabla_{can}) \right)$ of extensions of flat bundles:
\begin{align}\label{eq06}
0\to (F^*L^{-1}_1, \nabla_{can})\to (H, \nabla )\to (F^*L_1,\nabla_{can}
)\to 0.
\end{align}
There is a natural forgetful map
$$
\textrm{Ext}^1\left((F^*L_1,\nabla_{can}),(F^*L_1^{-1},\nabla_{can}) \right)\to \textrm{Ext}^1(F^*L_1,F^*L_1^{-1})
$$
by forgetting connections in an exact sequence of (\ref{eq06}). We let $B$ denote its image. Thus by the short exact sequence (\ref{eq01}), we see that $A\subset B$. Later we show that $A$ is an affine subspace of the ambient space $H^1(X_1,F^*L_1^{-2})$, not passing through the origin, and $B$ is the linear hull of $A$. On the other hand, if
$C^{-1}_{X_1\subset X_2}(E_1,\theta_1)$ also satisfies
(\ref{eq00}), then it admits a subsheaf $L_1\otimes \mathscr{L}\hookrightarrow
C ^{-1}_{X_1\subset X_2}(E_1,\theta_1)$ with $\mathscr{L}\in \mathrm{Pic}^0(X_1)$. For a given element $\xi\in H^1(X_1,F^*L_1^{-2})$, let $H_{\xi}$ be the corresponding isomorphism class of extension:
\begin{align}\label{eq04}
0\to F^*L_1^{-1}  \to  H_{\xi} \to F^*L_1 \to 0,
\end{align}
Thus this leads us also to considering the subset $K\subset H^1(X_1,F^*L_1^{-2})$, consisting of extensions such that $H_\xi$ admits a subsheaf $L_1\otimes \mathscr{L}\hookrightarrow H_\xi$ for some  $\mathscr{L}\in \text{Pic}^0(X_1)$. $K$ is clearly a cone. We shall call it the \emph{periodic cone} of $X_1/k$.
\begin{lemma}
The periodic cone $K$ can be written in the following form:
\begin{align}\label{eq05}
K=\bigcup_{s\in\mathbb{P}\text{Hom}(L_1\otimes \mathscr{L}, F^*L_1), \mathscr{L}\in \text{Pic}^0(X_1)}\mathrm{Ker}(\phi_s),
\end{align}
where the map on $H^1$
$$
\phi_s: H^1(X_1,F^*L_1^{-1}\otimes F^*L_1^{-1}) \to H^1(X_1, L_1^{-1}\otimes \mathscr{L}^{-1}\otimes F^*L_1^{-1})
$$ is induced from the sheaf morphism $\check s\otimes id$ with $\check{s}: F^*L_1^{-1}\to L_1^{-1}\otimes \mathscr{L}^{-1}$ the dual of $s$.
\end{lemma}
\begin{proof}
If $H_{\xi}$ admits a subsheaf $L_1\otimes \mathscr{L}$ with $\mathscr{L}\in \text{Pic}^0(X_1)$, then the composite $L_1\otimes \mathscr{L}\hookrightarrow  H_{\xi}\twoheadrightarrow F^*L_1$ cannot be zero since, otherwise, one gets a nonzero morphism $L_1\otimes \mathscr{L}\to F^*L_1^{-1}$ which contradicts the fact  $$\deg(L_1\otimes \mathscr{L})>0>\deg(F^*L_1^{-1}).$$ So for a chosen nonzero element $s\in \text{Hom}(L_1\otimes \mathscr{L}, F^*L_1)$, our task is to describe those extensions $H_{\xi}$ such that $s$ is liftable to a morphism $L_1\otimes \mathscr{L}\to H_{\xi}$. Pulling back the extension $H_{\xi}$ along $s$, one obtains an extension of $L_1\otimes \mathscr{L}$ by $F^*L_1^{-1}$ whose class is given by the image of $\xi$ under the map
$$
\phi_s: H^1(X_1,F^*L_1^{-1}\otimes F^*L_1^{-1}) \stackrel{\check s\otimes id}{\to} H^1(X_1, L_1^{-1}\otimes \mathscr{L}^{-1}\otimes F^*L_1^{-1}).
$$
It is clear that $s$ lifts to a morphism $L_1\otimes \mathscr{L}\to H_{\xi}$ iff the pull-back extension of $H_{\xi}$ along $s$ is split, that is, $\phi_s(\xi)=0$.
\end{proof}

Viewing $H^1(X_1, F^*L_1^{-2})$ as the affine space, we shall see that both $A$ and $K$ are actually closed subvarieties, see Propositions \ref{affine} and  \ref{cone}. One of the main points of the paper is to show that for a generic curve $X_1\in \mathscr{M}_{g,r}$, $A\cap K\neq \emptyset$. This fact has the consequence that one finds then a $W_2$-lifting $X_2$ and a Hodge filtration $Fil_1$ on $C^{-1}_{X_1\subset X_2}(E_1,\theta_1)$ such that the isomorphism (\ref{eq00}) holds, see Proposition \ref{prop20}. This is the way we prove Theorem \ref{thm03}. We use a degeneration argument to show $A\cap K\neq \emptyset$ for a generic curve. In fact, we consider a totally degenerate curve of genus $g$ with $r$-marked points in \S4, where we are able to prove the intersection $A\cap K$ consists of a unique element and consequently is nonempty. In \S5, we show that $A$s and $K$s form families when the base curve $X_1$ deforms and then, by upper semi-continuity, it follows that $A\cap K\neq \emptyset$ for closed points in a nonempty open neighborhood of a totally degenerate curve.

Now we turn to the lifting problem of a periodic Higgs bundle. The key is to resolve the obstruction class of lifting the Hodge filtration. For this purpose, we make the following
\begin{definition}\label{ordinary section}
 Assume that  $\mathscr{L} \in\text{Pic}^0(X_1)$.
We call  $s\in \text{Hom}(L_1\otimes\mathscr{L}, F^*L_1)$ ordinary if the composite \begin{align} \label{ordinary1}
H^1(X_1, L_1^{-2})\overset{F^*}\to H^1(X_1, F^*L_1^{-2})
\overset{\check{s}^2}\to H^1(X_1, L_1^{-2}\otimes \mathscr{L}^{-2})
\end{align}
is injective, where  the second map  is induced by the sheaf morphism $\check{s}^{2}$.
\end{definition}
This definition is equivalent to the one introduced by Mochizuki in \cite{Mo}, although it was stated in a totally different form. In our context, the ordinariness ensures the existence of lifting Hodge filtrations.
See \S\ref{lifting section}. Denote by $(E_n, \theta_n, Fil_n, \phi_n)$ the two-periodic Higgs bundle over $X_n$ in \eqref{tower}. This brings us to Theorem \ref{thm04}.

\begin{theorem}\label{thm04}
Use notation as in Theorem \ref{thm03}. Then with respect to a $W_2$-lifting $X_2$ of $X_1$, $(E_1,\theta_1)$ becomes a two-periodic Higgs bundle $(E_1, \theta_1, Fil_1, \phi_1)$. Let $s$ be the composite $Fil_1\hookrightarrow C_{1}^{-1}(E_{1},\theta_{1}) \twoheadrightarrow F^*L_1$,
where  $Fil_1$ is of the form $L_1\otimes \mathscr{L}$ for a two-torsion line bundle $\mathscr{L}$ and the second map is given in \eqref{eq02}. If $s$ is ordinary,
 then for all $n\geq 1$, inductively there exists a  $W_{n+1}$-lifting $X_{n+1}$ of $X_n$ such that the two-periodic Higgs bundle
 $(E_{n-1}, \theta_{n-1}, Fil_{n-1} , \phi_{n-1} )$  over $X_{n-1}$ can be lifted to
  a two-periodic Higgs bundle $(E_n, \theta_n, Fil_n , \phi_n )$ over $X_n$.
\end{theorem}

We are now in a position to give a simple proof of Theorem \ref{mainthm}.

\begin{proof}[Proof of Theorem \ref{mainthm}]
 With the ordinary condition  ensured by   Proposition \ref{ordinary} below, Theorem \ref{mainthm} is  a direct consequence of Corollary \ref{cor of thm 03} and Theorem \ref{thm04}.
\end{proof}

\section{The smooth case}
In this section, we begin to investigate the properties of the subsets $A$ and $K$ in the case where the base curve $X_1/k$ is a smooth projective curve of genus $g$ together with $r$-marked points $D_1\subset X_1$. It is required that $2g-2+r$ be an even positive number. To conform with the notation in later sections, we use the fine logarithmic structures defined in \cite{KKa} and notation therein. In particular, we equip $X_1$ with the standard log structure defined by the divisor $D_1$ (Example 1.5 (1) \cite{KKa}) and the base $\mathrm{Spec}\ k$ with the trivial log structure (so that the structural morphism $X_1\to \mathrm{Spec}\ k$ is log smooth Example 3.7 (1) \cite{KKa}); $\omega_{X_1/k}$ denotes the invertible sheaf of differential forms with respect to these log structures (1.7 \cite{KKa}).

\subsection{General discussion}
In the last section, we introduced the subsets $A$ and $K$ of the affine space $H^1(X_1,F^*L_1^{-2})$. We proceed to study some basic geometric properties of these two subsets and their intersection behavior.

Let $\sigma$ be the Frobenius automorphism of the field $k$. Then for any vector bundle $V$ over $X_1$, the absolute Frobenius of $X_1$ induces a $\sigma$-linear (i.e. semilinear) morphism $F^*: H^i(X_1,V)\to H^i(X_1,F^*V)$ for $i\geq 0$. Indeed, it is the composite of natural morphisms:
$$
H^i(X_1,V)\cong H^i(X_1,V\otimes_{\sO_{X_1}} \sO_{X_1})\to H^i(X_1,V\otimes_{\sO_{X_1}} F_*\sO_{X_1})\cong H^i(X_1,F^*V).
$$
In our case, we denote by $W_F$ the image of the natural map $F^*: H^1(X_1, L_1^{-2})\to H^1(X_1, F^*L_1^{-2})$. Then we have
\begin{proposition}\label{affine}
The subset $A\subset H^1(X_1, F^*L_1^{-2})$ is a \emph{nontrivial} translation of the linear subspace $W_F$. Moreover, $\dim A=3(g-1)+r$.
\end{proposition}
We postpone the proof.  We point out first that $W_F$ is actually a linear subspace of $B$. Indeed, one can also interpret the natural map $F^*: H^1(X_1,L_1^{-2})\to H^1(X_1,F^*L_1^{-2})$ in terms of extensions. That is, it is the map obtained by pulling back an extension of form
\begin{align}\label{eq09}
0\to L_1^{-1} \to E\to L_1\to 0
\end{align}
via the absolute Frobenius to an extension of form
\begin{align}\label{eq009}
0\to F^*L_1^{-1} \to F^*E\to F^*L_1\to 0.
\end{align}
Note that the latter extension is tautological to the extension equipped with the canonical connections:
\begin{align}\label{eq0009}
0\to (F^*L_1^{-1},\nabla_{can}) \to (F^*E,\nabla_{can})\to (F^*L_1,\nabla_{can})\to 0.
\end{align}
Thus, we see that $W_F$ lies in $B$. In fact, there is also a natural map $H^1(X_1,L_1^{-2})$ to $\textrm{Ext}^1\left((F^*L_1,\nabla_{can}),(F^*L_1^{-1},\nabla_{can}) \right)$. To see this, we let $H^1_{dR}:=H^1_{dR}(F^*L_1^{-2},\nabla_{can})$ be the first hypercohomology of the de Rham complex
$$
\Omega_{dR}^*(F^*L_1^{-2},\nabla_{can}):=F^*L_{1}^{-2}\stackrel{\nabla_{can}}{\longrightarrow} F^*L_1^{-2}\otimes \omega_{X_1/k}.
$$
There are natural morphisms of complexes
$$
L_1^{-2}\to \Omega_{dR}^*(F^*L_1^{-2},\nabla)
$$
and
$$
\Omega_{dR}^*(F^*L_1^{-2},\nabla)\to F^*L_1^{-2}.
$$
Identifying $\textrm{Ext}^1\left((F^*L_1,\nabla_{can}),(F^*L_1^{-1},\nabla_{can}) \right)$ with $H^1_{dR}$, one sees that the former morphism induces the map on hypercohomologies
$$
\beta: H^1(X_1,L_1^{-2})\to  \textrm{Ext}^1\left((F^*L_1,\nabla_{can}),(F^*L_1^{-1},\nabla_{can}) \right),
$$
while the latter induces
$$
\alpha:  \textrm{Ext}^1\left((F^*L_1,\nabla_{can}),(F^*L_1^{-1},\nabla_{can}) \right)\to H^1(X_1,F^*L^{-2}_1),
$$
which is just the forgetful map defining $B$. Moreover, the composite map
$$
\alpha\circ \beta: H^1(X_1,L_1^{-2})\to H^1(X_1,F^*L^{-2}_1)
$$
is the map $F^*$ defining $W_F$. An injectivity result is now in order.
\begin{lemma}\label{injective}
 Both the map $\alpha: H^1_{dR}(F^*L_1^{-2},\nabla_{can})\to H^1(X_1,F^*L_1^{-2})$ and the map $\alpha\circ\beta=F^*: H^1(X_1,L_1^{-2})\to H^1(X_1,F^*L^{-2}_1)$ are injective.
 \end{lemma}
 \begin{proof}
Consider first the map $\alpha$. We prove by contradiction. Assume the contrary. Then, there exists a nonsplit exact sequence (\ref{eq06}) which becomes split if ignoring the connections in the sequence. Let $F^*L_1\hookrightarrow H$ be a splitting of vector bundles. Since its image cannot be preserved by the connection $\nabla$, the associated Higgs bundle by taking grading with respect to the image must have nonzero Higgs field. That is,
$$
\bar \nabla: F^*L_1 \to (H/F^*L_1\cong F^*L^{-1}_1)\otimes \omega_{X_1/k}
$$
is nonzero. But, since
$$
\deg F^*L_1=\frac{p}{2}(2g-2+r)>\frac{2-p}{2}(2g-2+r)=\deg (F^*L^{-1}_1\otimes \omega_{X_1/k}),
$$
this is a contradiction. The injectivity for the map $F^*$ also follows. Indeed, the previous argument shows if the exact sequence (\ref{eq009}) is split then the exact sequence (\ref{eq0009}) with connections is also split. Then one applies the classical Cartier descent theorem (viz, not the logarithmic one) to conclude the splitting of (\ref{eq09}).
\end{proof}
\begin{remark}\label{general inject}
The above argument for $F^*$ shows the following injectivity statement: for a smooth projective variety $X/k$ and a line bundle $L$ over $X$ satisfying $\deg L>\frac{\mu_{\max}(\Omega_X)}{p}$, where $\mu_{\max}(\Omega_X)$ is the maximal slope of subsheaves of $\Omega_X$, the natural morphism $F^*: H^1(X, L^{-1})\to H^1(X,F^*L^{-1})$
is injective. It can be also directly deduced from the proof of Lemma 2.23 \cite{La}.
\end{remark}

By the above lemma, one computes the dimension of $W_F$ (and hence of $A$) by Riemann-Roch. The dimension of $H^1_{dR}$ is computed via its natural isomorphism to its corresponding Higgs cohomolgy \cite{DI},\cite{OV}. Define $H^1_{Hig}:=H^1_{Hig}(L_1^{-2},0)$ to be the hypercohomology of the Higgs complex
$$
\Omega^*_{Hig}(L_1^{-2},0):=L_1^{-2}\stackrel{0}{\to} L_1^{-2}\otimes \omega_{X_1/k}.
$$
Clearly,
$$
H^1_{Hig}\cong H^1(X_1,L_1^{-2})\oplus H^0(X_1,L_1^{-2}\otimes \omega_{X_1/k}\cong \sO_{X_1}).
$$
\begin{lemma}\label{dimension of B}
There is a natural isomorphism $$H^1_{dR}(F^*L_{1}^{-2},\nabla_{can})\cong H^1_{Hig}(L_1^{-2},0).$$ Therefore $\dim B= \dim A+1$.
\end{lemma}
\begin{proof}
See Corollary 2.27 \cite{OV} for the case $D_1=\emptyset$. However in the zero Higgs field case, one extends directly the splitting formula due to Deligne-Illusie \cite{DI} to construct an explicit quasi-isomorphism from the Higgs complex $\Omega^*_{Hig}(L_1^{-2},0)$ to the simple complex associated to the  \v{C}ech double complex of $F_*(\Omega^*_{dR}(F^*L_1^{-2},\nabla_{can}))$. See, for example, the proof of Theorem 10.7 \cite{EV}. In particular, the argument works also in this special log case.
\end{proof}

The proof of the following proposition will be postponed to a later subsection.
\begin{proposition}\label{cone}
The periodic cone $K$ is a closed subvariety of $H^1(X_1, F^*L_1^{-2})$. Moreover $\dim K=(p-1)(2g-2+r)$.
\end{proposition}
Next, we explain the intersection behavior of $A$ and $K$ inside the ambient space $H^1(X_1, F^*L_1^{-2})$. A simple dimension calculation shows
\begin{align}\label{eq07}
\dim A + \dim K = \dim H^1(X_1, F^*L_1^{-2}).
\end{align}
Let $\mathbb P:=\mathbb{P}(H^1(X_1, F^*L_1^{-2}))$ be the associated projective space and $$p: H^1(X_1, F^*L_1^{-2}) \backslash \{ 0\} \to \mathbb{P}$$ be the natural map. By Proposition \ref{affine}, $p(A)\subset \mathbb P$ is an affine space of the same dimension as $A$. Denote by $\mathbb{P}(K)$ the projectivized periodic cone, and similarly use notation $\mathbb{P}(W_F)$ and $\mathbb{P}(B)$. Now, by the above discussion, we get
\begin{align}\label{eq007}
\dim \mathbb{P}(B)+\dim \mathbb{P}(K)=\dim \mathbb{P}.
\end{align}
Hence $\mathbb{P}(B)$ and $\mathbb{P}(K)$ always intersect. For each
$\xi\in  \mathbb{P}(B)\cap \mathbb{P}(K)$, the corresponding extension $(H,\nabla)$ admits an invertible subsheaf $L_1\otimes\mathscr{L}$ for some
$\mathscr{L}\in \text{Pic}^0(X_1)$ and $\xi\in \mathrm{Ker}(\phi_s)$ with $s$ the composite map $L_1\otimes \mathscr{L} \hookrightarrow H\to F^*L_1$. This is almost what we need to prove the periodicity but not exactly: the next proposition shows that those $\xi$s in the subset $p(A)\cap \mathbb{P}(K)$, if nonempty, give rise to periodicity. But clearly
$$
p(A)\cap \mathbb{P}(K)=\mathbb{P}(B)\cap \mathbb{P}(K)-\mathbb{P}(W_F)\cap \mathbb{P}(K).
$$

\begin{proposition}\label{prop20}
If $p(A)\cap \mathbb{P}(K)\neq\emptyset$ or equivalently $A\cap K\neq \emptyset$, then $(E_1, \theta_1)$ is a two-periodic Higgs bundle, that is, there exists a $W_2$-lifting $X_2$ of $X_1$ and a Hodge filtration $Fil_1$  such that
\begin{align}
&Gr_{Fil_1}\circ C ^{-1}_{X_1\subset X_2}(E_1,\theta_1)\cong (E_1,\theta_1)\otimes (\mathscr{L},0),
 \end{align}
with $\mathscr{L}$ a certain two-torsion line bundle.
\end{proposition}

\begin{proof}
Take $\xi\in A\cap K$. By definition, the corresponding extension is of the form $H_\xi=C ^{-1}_{X_1\subset X_2}(E_1, \theta_1)$ for some $W_2$-lifting $X_2$ of $X_1$ and $H_\xi$ has a subsheaf  $L_1\otimes \mathscr{L}\hookrightarrow H_\xi$ for $\mathscr{L}\in \text{Pic}^0(X_1)$.

Next, we show that the subsheaf $L_1\otimes \mathscr{L}\hookrightarrow H_\xi$
is saturated and not $\nabla$-invariant.  Let $Fil_1$ be
the saturated subbundle of  $L_1\otimes \mathscr{L}\hookrightarrow H_\xi$. We claim that  $Fil_1$ is not $\nabla$-invariant. Otherwise,
since the $p$-curvature of $\nabla$ on $H_\xi$ is nilpotent,
the  $p$-curvature of $\nabla$ on $Fil_1$  would be
zero; as $H_\xi$ has another subbundle $F^*L_1^{-1}$ with trivial
$p$-curvature, the $p$-curvature of $H_\xi$ would have to be zero. But this is a contradiction since $H_\xi$ is the inverse Cartier of a Higgs bundle with maximal Higgs field and its $p$-curvature is therefore nonzero (actually nowhere zero).

Having now established that $Fil_1$ is not $\nabla$-invariant, one gets a nonzero Higgs field
$$
\bar \nabla: (Fil_1)\to (Fil_1)^{-1}\otimes \omega_{X_1/k}.
$$
Thus one obtains relations:
$$
\frac{2g-2+r}{2}=\deg (L_1\otimes \mathscr{L})\leq \deg Fil_1\leq \frac{1}{2}\deg \omega_{X_1/k}=\frac{2g-2+r}{2}.
$$
It follows immediately that $Fil_1\cong L_1\otimes
\mathscr{L}$, the Higgs field $\bar \nabla$ is an
isomorphism, and $\mathscr{L}$ is a two-torsion. To summarize, one gets
$$
Gr_{Fil_1}\circ C_{X_1\subset X_2}^{-1}(E_1,\theta_1)\cong (E_1,\theta_1)\otimes (\mathscr{L},0)
$$
as claimed.
\end{proof}
The last proposition reveals the intimate relation of the intersection $A\cap K$ and the periodicity of the Higgs bundle $(E_1,\theta_1)$ with maximal Higgs field. In the case of $\mathbb{P}^1$ with four marked points, one can directly show $A\cap K\neq \emptyset$. In general, we can only show the nonempty intersection for a generic $X_1$ in the moduli using a degeneration argument.

\begin{proposition}\label{proj line}
Let $X_1$ be a $\mathbb{P}^1$ with four marked points. Then the Higgs bundle $(E_1=\sO(1)\oplus \sO(-1),\theta_1)$ of type (\ref{Higgsbdl}) is one-periodic.
\end{proposition}
\begin{proof}
By the above discussion it suffices to show that $\mathbb{P}(W_F)\cap \mathbb{P}(K)=\emptyset$, or equivalently $W_F\cap K=0$. First notice that the only nontrivial extension of $\sO(-1)$ by $\sO(1)$ over $\mathbb{P}^1$ is isomorphic to $\sO^{\oplus 2}$. Therefore, $H_{\xi}$ corresponding to a nonzero $\xi\in W_F$ is also isomorphic to $\sO^{\oplus 2}$. It follows that $H_\xi$ does not admit any invertible subsheaf of positive degree, which implies $\xi\notin K$ and then $W_F\cap K=0$.
\end{proof}

\subsection{Proof of Proposition \ref{affine}}
Let $\check {\theta}_1:  \mathcal{T}_{X_1/k}\stackrel{\cong}{\to} L_1^{-2}$ be the isomorphism induced by the Higgs field. By abuse of notation, we use it to also denote the induced isomorphism on cohomologies. The following diagram commutes:
\begin{align}\label{eq201}
\xymatrix{
H^1(X_1,  \mathcal{T}_{X_1/k}) \ar[d]^{F^*} \ar[rr]^-{ \check{\theta}_1} && H^1(X_1, L_1^{-2}) \ar[d]^{F^*} \\
H^1(X_1, F^*  \mathcal{T}_{X_1/k}) \ar[rr]^{F^*(\check{\theta}_1)} &&H^1(X_1,F^* L_1^{-2}). }
\end{align}

Proposition \ref{affine} follows from  the following observations.
\begin{lemma}\label{observs}
\begin{enumerate}
\item
Consider the map $\rho: \{X_1\subset X_2 \} \to H^1(X_1,F^*L_1^{-2})$ defined in (\ref{eq03}). Then for $\tau\in \{X_1\subset X_2 \}$ and $\nu \in H^1(X_1,  \mathcal{T}_{X_1/k})$, it holds that
\begin{align}\label{eq203}
\rho(\tau+\nu)=\rho(\tau)-F^*\circ\check{\theta}_1(\nu).
\end{align}
\item The image of $\rho$ does not pass through the origin of $H^1(X_1, F^*L_1^{-2})$.
\end{enumerate}
\end{lemma}
\begin{proof}
(1) Given a $W_2$-lifting $X_2$ of $X_1$, the obstruction to lifting the absolute Frobenius over $X_2$ lies in $H^1(X_1,F^*\mathcal{T}_{X_1/k})$. Therefore, one obtains the following map
$$
\mu: \{ X_1\subset X_2\} \to H^1(X_1,F^*  \mathcal{T}_{X_1/k}).
$$

Using the construction of the inverse Cartier transform of Ogus-Vologodsky via exponential twisting (see \cite{LSZ0} and see also Appendix in the log setting), one sees immediately that $\rho$ is nothing but the composite $F^*(\check{\theta}_1)\circ \mu$. Via the commutativity of the diagram (\ref{eq201}), it suffices to show that $\mu$ is a torsor map under the map
$$
F^*: H^1(X_1, \mathcal{T}_{X_1/k})\to H^1(X_1,F^*\mathcal{T}_{X_1/k}).
$$
This statement should be known in the literature, although we cannot find a suitable reference. In the following, we shall provide a proof for the reader's convenience. This proof might also help the reader to understand the  computations in the proof of Lemma \ref{cohclass}.

Let $X_2$ and $\hat X_2$ be two $W_2$-liftings of the log curve $X_1$. Let
$\sU=\{U_i\}_{i\in I}$ be an open affine covering of $X_1$. Following a suggestion of the referee, we assume additionally that the boundary divisor $D_1$ does not intersect with any overlap $U_{ij}:=U_i\cap U_j$. By doing so, we can avoid the log differential calculus to define the cocyle $\{h_{ij}\}$ below. Note that there is a unique $W_2$-lifting $V_i$ of $U_i$ up to isomorphism (which we shall suppress in our computations). Let $V_{ij}$ be a
$W_2$-lifting of $U_{ij}$. Fix the embeddings $V_{ij}\to V_j$.
Corresponding to liftings $X_2$ and $\hat X_2$, there are two embeddings $g_{ij}: V_{ij}\to V_i$, and $\hat g_{ij}:V_{ij}\to V_i$. Note that $g_{ij}$ and $\hat g_{ij}$ have the same reduction $U_{ij}\to U_i$ modulo $p$. Then the image of  $g_{ij}^*-(\hat g_{ij})^*: \sO_{V_{i}}\to \sO_{V_{ij}}$ is contained in $p\sO_{V_{ij}}$. Therefore, one obtains the morphism $\frac{g_{ij}^*-(\hat g_{ij})^*}{p}:\sO_{U_{i}}\to \sO_{U_{ij}}$ which is known to be a $k$-derivation. Let $\tilde \nu_{ij}: \Omega_{U_i/k}\to \sO_{U_{ij}}$ be the $\sO_{U_i}$-morphism such that $\tilde \nu_{ij}\circ d=\frac{g_{ij}^*-(\hat g_{ij})^*}{p}$. Localizing $\tilde \nu_{ij}$ over $U_{ij}$, one obtains an $\sO_{U_{ij}}$-morphism $ \omega_{U_{ij}/k}=\Omega_{U_{ij}/k}\to \sO_{U_{ij}}$ which can be regarded as a local section $\nu_{ij}$ of the sheaf $\sT_{X_1/k}$ over $U_{ij}$. It is clear that $\{\nu_{ij}\}$ is a 1-cocycle which represents the cohomology class $\nu=[X_2]-[\hat X_2]\in H^1(X_1,\mathcal{T}_{X_1/k})$. The well-definedness of the cohomology class is implicitly contained in \S8 \cite{EV}.

On the other hand, let us choose and then fix log Frobenius liftings $\{F_i:V_i\to V_i\}_{i\in I}$. For simplicity, let us assume $F_i$ restricts a morphism $V_{ij}\to V_{ij}$. Thus we get the composite morphisms
$$
g_{ij}\circ F_j: V_{ij}=V_{ji}\stackrel{F_j}{\to} V_{ji}=V_{ij}\stackrel{g_{ij}}{\to}V_i
$$
and
$$
F_i\circ g_{ij}: V_{ij}\stackrel{g_{ij}}{\to} V_{i}\stackrel{F_i}{\to}V_{i}.
$$
Both are log Frobenius liftings over $V_{ij}$. Set
$$
h_{ij}:= \frac{(F_i\circ g_{ij})^*-(g_{ij}\circ F_j)^*}{p}:
F^*\sO_{U_{i}}\to \sO_{U_{ij}}.
$$
The 1-cocyle condition for $\{h_{ij}\}$ as well as its well-definedness up to coboundary are implicitly contained in Corollary 9.12 \cite{EV}. Thus $\{h_{ij}\}$ gives rise to a \v{C}ech-representative of the class $\mu(X_1\subset X_2)$. Replacing $g_{ij}$ by $\hat g_{ij}$, we obtain similarly $\{\hat h_{ij}:=\frac{(F_i\circ \hat g_{ij} )^*-(\hat g_{ij}\circ F_j)^*}{p}\}$ for $\mu(X_1\subset \hat X_2)$. Thus it suffices to show the following equality:
$$
h_{ij}-\hat h_{ij}=-\frac{F^*\circ( g_{ij})^*-F^*\circ \hat g_{ij}^*}{p}.
$$
The left hand side is equal to
$$
\frac{(g_{ij}^*-(\hat g_{ij})^*)\circ F_i^*}{p}-\frac{ F^*\circ
(g_{ij}^*-(\hat g_{ij})^*)}{p}.
$$
It suffices to show the first term is zero. But this is clear: since $\frac{g_{ij}^*-(\hat g_{ij})^*}{p}$ is a $k$-derivation, for any element $\alpha\in
\sO_{U_i}$ over $U_{ij}$,
$$
\frac{(g_{ij}^*-(\hat g_{ij})^*)\circ F_i^*}{p}(\alpha\otimes 1)=\tilde\nu_{ij}\circ d(F^*\alpha)=0.
$$
Therefore, the required equality follows.

(2) Assume the contrary, namely, for some $W_2$-lifting, the extension (\ref{eq02}) given by the bundle part of the inverse Cartier transform of $(E_1,\theta_1)$ splits. By Lemma \ref{injective}, the extension as flat bundles (\ref{eq01}) splits as well. It follows that its $p$-curvature is zero, which is however impossible.
\end{proof}

\subsection{Proof of Proposition \ref{cone}}\label{consec}
To exhibit the algebraic structure of the periodic cone $K$, we proceed as follows. Note that the notions of locally free sheaves of finite rank and vector bundles will be used interchangeably in our description.\\

{\itshape Step 1}: Denote by $p_1: \text{Pic}^0(X_1)\times X_1 \to \text{Pic}^0(X_1)$ and  $p_2: \text{Pic}^0(X_1)\times X_1 \to   X_1$ the projections. Write $\sO=\sO_{\text{Pic}^0(X_1)\times X_1}$. Let $\mathcal{L}$ be the universal line bundle over $\text{Pic}^0(X_1)\times X_1$. Consider two invertible sheaves $p_2^*F^*L_1^{-1}$ and $p_2^{*}L_1^{-1}\otimes \sL^{-1}$ which form the sheaf
$\mathcal{H}om_{\sO}(p_2^*F^*L_1^{-1},p_2^{*}L_1^{-1}\otimes \sL^{-1})$.
Its direct image $p_{1*}\mathcal{H}om_{\sO}(p_2^*F^*L_1^{-1},p_2^{*}L_1^{-1}\otimes \sL^{-1})$ is a locally free $\sO_{\text{Pic}^0(X_1)}$-modules of rank $\frac{p-1}{2}(2g-2+r)-(g-1)$ by Riemann-Roch. Let $\V$ be the vector bundle $p_{1*}\mathcal{H}om(p_2^*F^*L_1^{-1},p_2^{*}L_1^{-1}\otimes \sL^{-1})$ and $\mathbb P$ its associated projective bundle. Let $\pi_0: \V\to \text{Pic}^0(X_1)$ be the natural projection and $\V^0=\V-\{0\}$, where $\{0\}$ is the zero section of $\pi_0$. Let $\bar{\pi}_0: \P\to \text{Pic}^0(X_1)$ be the natural projection. Set $\pi=\pi_0\times id$. The composite  
$$
g_2:\V\times X_1\stackrel{\pi}{\to}\text{Pic}^0(X_1)\times X_1  \stackrel{p_2}{\to} X_1
$$
is simply the natural projection to the second factor. Set $g_1: \V\times X_1\to \V$ to be the first projection. \\

\iffalse

The following commutative diagram will be used below:
$$
\xymatrix{
 \V\times X_1   \ar[r]^{\pi_0\times id}                    &  \text{Pic}^0(X_1)\times X_1     \ar[r]^{ \qquad   p_2}&   X_1\\
 (\V-\{0\}) \times X_1 \ar[r]^{q_0\times id}\ar[u]   \ar[ur]^{\pi^0} &  \P\times X_1 \ar[r]^{g_1}   \ar[u]^{\bar \pi_0\times id}\ar[ur]^{u} &  \P, }
$$
where the left vertical map is the natural inclusion and $g_1$ is the natural projection. Set $\pi=\pi_0\times id$, $\eta=\bar \pi_0\times id$,  $q=q_0\times id$ and $q_1=g_1\circ q$.
\fi
{\itshape Step 2}: By construction, there is a tautological morphism
$$
t: \pi^*p_2^*F^*L_1^{-1}\to \pi^*p_2^*L_1^{-1}\otimes \pi^*\sL^{-1}.
$$
Consider further
$$
id\otimes t: \pi^*p_2^*F^*L_1^{-1}\otimes  \pi^*p_2^*F^*L_1^{-1}\to\pi^*p_2^*F^*L_1^{-1}\otimes \pi^*p_2^*L_1^{-1}\otimes \pi^*\sL^{-1},
$$
that induces the following morphism on the higher direct images:
$$
\phi: R^1g_{1,*}\pi^{*}p_2^*F^*L_1^{-2}\to R^1g_{1,*}\pi^{*}(p_2^*F^*L_1^{-1}\otimes p_2^*L_1^{-1}\otimes \sL^{-1}).
$$
Let $G$ denote $R^1g_{1,*}\pi^{*}p_2^*F^*L_1^{-2}$ and $H$ denote $R^1g_{1,*}\pi^{*}(p_2^*F^*L_1^{-1}\otimes p_2^*L_1^{-1}\otimes \sL^{-1})$. Then, the inverse image $\mathrm{Ker}(\phi)$ under $\phi$ of the zero section of the bundle $H$ is a closed subset of $G$.\\

{\itshape Step 3}: Note that $G=R^1g_{1,*}g_{2}^{*}(F^*L_1^{-2})\cong H^1(X_1,F^*L_1^{-2})\times \V$ as a vector bundle over $\V$. Set $\mathrm{Ker}(\phi)^0=\mathrm{Ker}(\phi)\cap G^0$, where
$G^0=H^1(X_1,F^*L_1^{-2})\times \V^0$ is an open subset of $G$. Notice that the closed subset $\mathrm{Ker}(\phi)^0\subset H^1(X_1,F^*L_1^{-2})\times \V^0$ is invariant under the natural $\G_m$-action on the second factor and hence descends to a closed subset of $H^1(X_1,F^*L_1^{-2})\times \P$. Indeed, let us consider the fiber $\phi_s: G_s\to H_s$ of the bundle morphism $\phi$ at a closed point $s\in \V^0$. Note for any $\lambda\in \G_m$, $\lambda s$ and $s$ in $\V^0$ have the same image $[\mathscr{L}]\in \text{Pic}^0(X_1)$ under $\pi_0$. Thus the fiber $H_{\lambda s}$ is canonically identified with $H_{s}$ which are both identified with $H^{1}(X_1, F^*L_1^{-1}\otimes L_1^{-1}\otimes \mathscr{L})$. Then, by construction, it holds that
$$
\phi_{\lambda s}(v)=\lambda \phi_s(v), \ \forall{v}\in H^1(X_1,F^*L_1^{-2}).
$$
Thus $\mathrm{Ker}(\phi)^0(v,\lambda s)=\mathrm{Ker}(\phi)^0(v,s)$ for all $v\in  H^1(X_1,F^*L_1^{-2})$ and all $\lambda\in \G_m$. By abuse of notation, we denote the closed subset in $H^1(X_1,F^*L_1^{-2})\times \P$ again by $\mathrm{Ker}(\phi)$ in the following, which shall not cause confusion. Now we consider the projections
\[\pi_1: \mathrm{Ker}(\phi) \to H^1(X_1, F^*L_1^{-2})  ; \quad \pi_2: \mathrm{Ker}(\phi) \to  \mathbb{P}. \]
Then since $\pi_1$ is proper, the image of $\mathrm{Ker}(\phi)$ under $\pi_1$ is a closed subset of $H^1(X_1, F^*L_1^{-2})$. By our construction, $\pi_1(\mathrm{Ker}(\phi))$ is nothing but the periodic cone $K$. \\

Next, we determine the dimension of $K$. It is easy to calculate the dimension of $\mathrm{Ker}(\phi)$ using the second projection $\pi_2$. Take a closed point $x=([s],[\mathscr{L}])\in \mathbb{P}$ with $[\mathscr{L}]=\bar \pi_0(x)\in \text{Pic}^0(X_1)$ and $s\in \text{Hom}(L_1\otimes \mathscr{L},F^*L_1)$. The fiber $\pi_2^{-1}(x)\subset \mathrm{Ker}(\phi)$ is naturally identified with $\mathrm{Ker}(\phi_s)\subset H^1(X_1, F^*L_1^{-2})$ defined before, where $\phi_s: H^1(X_1,F^*L_1^{-2}) \to H^1(X_1, L_1^{-1}\otimes \mathscr{L}^{-1}\otimes F^*L_1^{-1}) $ is the morphism induced by $s$. Let $\text{Div}(s)$ be the zero divisor of the morphism $L_1\otimes \mathscr{L}\overset{s}\hookrightarrow F^*L_1$. Then one computes that
\begin{align}
\dim \mathrm{Ker}(\phi_s) = \deg(\text{Div}(s))=\frac{p-1}{2}(2g-2+r).
\end{align}
Therefore,
\begin{align}
\dim \mathrm{Ker}(\phi)&=\dim \mathbb{P}  +\dim \mathrm{Ker}(\phi_s)\\
&=(p-1)(2g-2+r).  \nonumber
\end{align}
Note also that $\mathrm{Ker}(\phi)$ is actually a vector bundle over $\mathbb P$ and therefore irreducible. The following lemma asserts that $\pi_1: \mathrm{Ker}(\phi)\to K$ is injective over a nonempty open subset and therefore $K$ is irreducible and of the same dimension as $\mathrm{Ker}(\phi)$.
\begin{lemma}\label{nointer}
For every closed point $x=([s], [\mathscr{L}])\in \mathbb{P}(p_{1*} (L_1^{-1}\otimes \mathcal{L}^{-1}\otimes F^*L_1)) $, the intersection
$$
\mathrm{Ker}(\phi_s) \cap (\bigcup_{ ([s'],[\mathscr{L}'])\neq ([s], [\mathscr{L}])\in \mathbb{P}(p_{1*} (L_1^{-1}\otimes \mathcal{L}^{-1}\otimes F^*L_1))}\mathrm{Ker}(\phi_{s'}))
$$
in the ambient space $H^1(X_1,F^*L_1^{-2})$ is contained in a proper closed subset of $\mathrm{Ker}(\phi_s)$. Therefore, the map $\pi_1$ is injective over any closed point of $\mathrm{Ker}(\phi_s)$ away from the above intersection.
\end{lemma}
\begin{proof}
Assume that $\mathrm{Ker}(\phi_s) \cap  \mathrm{Ker}(\phi_{s'})\neq 0$ for some $([s'],[\mathscr{L}']) \in \mathbb{P}$. Note that $\mathrm{Ker}(\phi_s) \cap  \mathrm{Ker}(\phi_{s'})$
is the kernel of the map
\[  \text{Ext}^1(F^*L_1, F^*L_1^{-1}) \overset{(\phi_s, \phi_{s'})}\longrightarrow \text{Ext}^1(L_1\otimes \mathscr{L}, F^*L_1^{-1})\oplus \text{Ext}^1(L_1\otimes \mathscr{L}', F^*L_1^{-1}),\]
which is induced by the sheaf morphism
\begin{align} \label{sub}
 F^*L_1^{-2} \overset{(s, s')}\longrightarrow  (L_1^{-1}\otimes \mathscr{L}^{-1}\otimes F^*L_1^{-1}) \oplus  (L_1^{-1}\otimes \mathscr{L}'^{-1}\otimes F^*L_1^{-1}).
\end{align}
A moment of thought gives us the saturation of the image of (\ref{sub}), which is $ F^*L_1^{-2}\otimes\mathcal{O}(D_{s, s'})$ with $D_{s, s'}=\text{Div}(s)\cap \text{Div}(s')$. Therefore the morphism (\ref{sub}) factors as
\[ F^*L_1^{-2}\hookrightarrow  F^*L_1^{-2}\otimes\mathcal{O}(D_{s, s'})   \hookrightarrow  (L_1^{-1}\otimes \mathscr{L}^{-1}\otimes F^*L_1^{-1}) \oplus  (L_1^{-1}\otimes \mathscr{L}'^{-1}\otimes F^*L_1^{-1}). \]
The long exact sequence of cohomologies of the following short exact sequence of vector bundles
\begin{align*}
 0\to & F^*L_1^{-2}\otimes\mathcal{O}(D_{s, s'}) \to  (L_1^{-1}\otimes \mathscr{L}^{-1}\otimes F^*L_1^{-1}) \oplus  (L_1^{-1}\otimes \mathscr{L}'^{-1}\otimes F^*L_1^{-1}) \\
\to & L_1^{-2}\otimes \mathscr{L}^{-1} \otimes \mathscr{L}'^{-1} \otimes \mathcal{O}(-D_{s, s'}) \to 0
\end{align*}
gives the injective map
$$H^1(X_1, F^*L_1^{-2}\otimes\mathcal{O}(D_{s, s'}) ) \to H^1(X_1,  L_1^{-1}\otimes \mathscr{L}^{-1}\otimes F^*L_1^{-1}) \oplus H^1(X_1, L_1^{-1}\otimes \mathscr{L}'^{-1}\otimes F^*L_1^{-1} ),
$$
as $H^0(X_1,  L_1^{-2}\otimes \mathscr{L}^{-1} \otimes \mathscr{L}'^{-1} \otimes \mathcal{O}(-D_{s, s'}))=0$ because $L_1^{-2}\otimes \mathscr{L}^{-1} \otimes \mathscr{L}'^{-1} \otimes \mathcal{O}(-D_{s, s'})$ is of negative degree. Therefore, $\mathrm{Ker}(\phi_s) \cap  \mathrm{Ker}(\phi_{s'})$ is the kernel of the map
\[ H^1(X_1, F^*L_1^{-2}  )  \to H^1(X_1, F^*L_1^{-2}\otimes\mathcal{O}(D_{s, s'}) ), \]
 which turns out to be the image of the injection
\[ H^0(X_1, \mathcal{O}_{D_{s, s'}} ) \hookrightarrow  H^1(X_1, F^*L_1^{-2}  ).\]
Here $H^0(X_1, F^*L_1^{-2}\otimes\mathcal{O}(D_{s, s'}) )=0$ again because of the degree. Clearly, $$\dim (\mathrm{Ker}(\phi_s) \cap  \mathrm{Ker}(\phi_{s'}))=\text{deg}(D_{s, s'}) <\text{deg}(\text{Div}(s))=\dim \mathrm{Ker}(\phi_s).$$
The above argument actually proves that
\begin{align*}
&\mathrm{Ker}(\phi_s) \cap (\bigcup_{ ([s'],[\mathscr{L}'])\neq ([s], [\mathscr{L}])\in \mathbb{P}(p_{1*} (L_1^{-1}\otimes \mathcal{L}^{-1}\otimes F^*L_1))}\mathrm{Ker}(\phi_{s'})) \\
\subset  & \bigcup_{D'\subsetneq\text{Div}(s)}\text{Im}(H^0(X_1, \mathcal{O}_{D'} ) \hookrightarrow  H^1(X_1, F^*L_1^{-2}  ) ),
\end{align*}
while the latter is a finite union of closed subsets of dimension strictly less than $\dim \mathrm{Ker}(\phi_s)$. This completes the proof.
\end{proof}

In order to employ a degeneration argument in the next section, we need to consider the twisted Higgs bundle $(\omega_{X_1/k}\oplus \sO_{X_1},\theta_1)=(E_1,\theta_1)\otimes (L_1,0)$. The final remark shows that the intersection behaviour of the set $A\cap K$ is invariant under this twist.
\begin{remark}\label{twist}
In the study of $A$ and $K$, we can replace $(E_1,\theta_1)$ of (\ref{Higgsbdl}) with one of its twists $(E_1,\theta_1)\otimes (M,0)$ by a line bundle $M$ equipped with the zero Higgs field. Indeed, first one gets the same ambient space
$$
\textrm{Ext}^1(F^*L_1\otimes F^*M,F^*L_1^{-1}\otimes F^*M)=\textrm{Ext}^1(F^*L_1,F^*L_1^{-1}),
$$
and then the same affine subspace $A$ because of the canonical isomorphism $$
C^{-1}_{X_1\subset X_2}((E_1,\theta_1)\otimes (M,0))\cong C^{-1}_{X_1\subset X_2}(E_1,\theta_1)\otimes (F^*M,\nabla_{can}).
$$
We also get the same periodic cone $K$ when one considers those extensions of $F^*L_1^{-1}\otimes F^*M$ by $F^*L_1\otimes F^*M$ containing the invertible subsheaf $L_1\otimes F^*M$.
\end{remark}

\section{The degeneration argument}
For a given curve $X_1\in \mathscr{M}_{g,r}$, where $2g-2+r$ is kept as before to be a positive and even number, it seems hard in general to determine whether or not $A$ intersects with $K$. Dimension counting (\ref{eq07}) makes the problem more interesting. The principal aim of this section is to show $A\cap K\neq \emptyset$ for a \emph{generic} curve via a degeneration argument. Therefore, we shall consider the coarse moduli space of stable curves $\overline{\mathscr{M}}_{g,r}$ and investigate the intersection behavior of $A$ and $K$ in a neighborhood of a so-called \emph{totally degenerate $(g,r)$-curve} (abbreviated as t.d. $(g,r)$-curve). The advantage of considering a t.d curve is that one can exhibit explicitly the nonemptiness of $A\cap K$. In fact, Proposition \ref{AKinter} asserts that it consists of exactly one point (multiplicity not counted). As another output of this method, we also show the existence of ordinary sections for a generic curve, which is essential in our method to guarantee the liftability of the periodic Higgs bundle to an arbitrary truncated Witt ring in \S\ref{lifting section}.

\subsection{The relative $A$ and $K$ over a local base}\label{section on relative A and K}
Recall that a t.d. $(g,r)$-curve is a nodal curve with $r$-marked points whose arithmetic genus is $g$ and normalization consists of disjoint copies of $\P^1$ with three marked points, say $\{0,1,\infty\}$ (the inverse image of nodes are also considered as marked points on the each component of the normalization). In the moduli $\overline{\mathscr{M}}_{g,r}$ of stable $(g,r)$-curves, there always exists such a t.d. $(g,r)$-curve.

Let $\chi$ be a t.d. $(g,r)$-curve. Take a flat family $f: \mathscr{X}_1\to B_1$ of stable $(g,r)$-curves parameterized by an affine curve $B_1$ such that its fiber over the closed point $0\in B_1$ is isomorphic to $\chi$ and such that its image in $\overline{\mathscr{M}}_{g,r}$ intersects transversally with the boundary divisor consisting of degenerate curves at the point $[\chi]$. Let $\tilde N$ be the log structure on $B_1$ associated to the divisor $0$ and $\tilde M$ the log structure on $\mathscr{X}_1$ associated to the divisor which is the sum of the vertical divisor $f^{-1}(0)$ and the horizontal divisor consisting of $r$-marked points. With these log structures, we get a log smooth morphism
$\tilde f: (\mathscr{X}_1,\tilde M)\to (B_1,\tilde N)$. Let $\Spec \ k\to B_1$ be the closed point $0$. Then we equip the structural morphism $f_0: \chi\to \Spec \ k$ with the pull-back log structure $\tilde f_0: (\chi, M)\to (\Spec \ k, N)$ of $\tilde f$ along the above map. Note that $ (\Spec \ k, N)$ is the so-called logarithmic point, and \'{e}tale locally around a node of $\chi$ the log structure $M$ admits a chart of the following form
$$
\xymatrix{
  \N^2 \ar[rr]^{\alpha_2} & & k[x,y]/(xy) \\
  \N \ar[u]^{{\rm diag}}\ar[rr]^{\alpha_1}& & \ar[u]_{}k,   }
$$
where $\alpha_1$ sends $1$ to $0$, $\alpha_2$ sends $(1,0)$ to $x$ and $(0,1)$ to $y$, and ${\rm diag}$ is the diagonal morphism. Whenever the context is clear, we may also use $f$ to denote $\tilde f$ by abuse of notation.
\iffalse
Let $\tilde f_0: (\chi,M)\to (\Spec \ k, N)$ be the associated log curve (where $(\Spec\ k, N)$ is the so-called logarithmic point). We shall also use the notion of a \emph{log smooth deformation} due to F. Kato (see Definition 8.1 \cite{FKa0}). In the following, we shall consider only the following type of log smooth deformations of the stable log curve $\tilde f_0$ if not otherwise specified: let $B_1$ be a smooth affine curve over $k$ equipped with the log structure by one closed point $0\in B_1$. Let $(B_1,\tilde N)$ be the log scheme. Let $\tilde f: (\mathscr{X}_1,\tilde M)\to (B_1,\tilde N)$ be a log smooth deformation of $f_0$ such that the underlying scheme of the fiber of $\tilde f$ over any other closed point of $B_1$ than 0 is a smooth $(g,r)$-curve.   Then the theory in \cite{FKa} makes this family into a log smooth deformation of $f_0$ with the required type. For a log smooth deformation $\tilde f$, let $f: \mathscr{X}_1\to B_1$ be the underlying morphism of schemes. \fi

One of the main advantages of log smooth deformations is that the sheaf of relative log differential forms $\omega_{\mathscr{X}_1/B_1}$ (w.r.t. the given log structures on $\mathscr{X}_1$ and $B_1$) is \emph{locally free} by Proposition 3.10 \cite{KKa}. Let $\sT_{\mathscr{X}_1/B_1}$ be its $\sO_{\mathscr{X}_1}$-dual. We shall consider the following relative Higgs bundle
\begin{align}
(\sE, \theta)=(\omega_{ \mathscr{X}_1/B_1}\oplus\sO_{ \mathscr{X}_1}, \theta),
\end{align}
where $\theta: \omega_{ \mathscr{X}_1/B_1} \to \sO_{\mathscr{X}_1}\otimes\omega_{ \mathscr{X}_1/B_1}$ is the natural isomorphism. For a closed point $b\neq 0$, it restricts to a Higgs bundle (\ref{Higgsbdl}) over $\mathscr{X}_{1,b}:=f^{-1}(b)$ twisted by $L_1$, i.e,
$$
(L_1^{2} \oplus \sO_{X_1},\theta_1), \quad \theta_1: L_1^2\stackrel{\cong}{\longrightarrow} \sO_{X_1}\otimes \omega_{X_1/k}.
$$

\begin{lemma}
The coherent sheaf of $\sO_{B_1}$-modules $R^1f_*F_{\mathscr{X}_1}^*
\sT_{\mathscr{X}_1/B_1}$ is locally free and of rank $(2p+1)(g-1)+pr$.
\end{lemma}
\begin{proof}
For each closed point $b\in B_1$, the fiber $(F_{\mathscr{X}_1}^*
\sT_{\mathscr{X}_1/B_1})_b$ is naturally isomorphic to $F^*\sT_{\mathscr{X}_{1,b}/k}$. We have shown that for $b\neq 0$, the dimension of $h^1(F^*\sT_{\mathscr{X}_{1,b}/k})$ is equal to $(2p+1)(g-1)+pr$. By the theorem of Grauert, our problem is reduced to showing that it is so for the fiber over $0$, which is Proposition \ref{dg_of_A} (1) below.
\end{proof}

Now we shall construct two closed subsets $\sA$ and $\sK$ inside the vector bundle $R^1f_*F_{\mathscr{X}_1}^*\sT_{\mathscr{X}_1/B_1}$ over $B_1$, such that over a closed point $0\neq b\in B_1$, they specialize to the $A$ and $K$ given before. \\

{\itshape Construction of $\sA$:} Let
$$\nabla_{can}: F_{\mathscr{X}_1}^*
\sT_{\mathscr{X}_1/B_1}\to F_{\mathscr{X}_1}^*
\sT_{\mathscr{X}_1/B_1}\otimes \omega_{\mathscr{X}_1/B_1}
$$
be the relative flat bundle. From it, one gets the relative de-Rham complex
$$
\Omega_{dR}^*:=\Omega_{dR}^*(F_{\mathscr{X}_1}^*
\sT_{\mathscr{X}_1/B_1},\nabla_{can}).
$$
There are natural morphisms of complexes $\Omega_{dR}^*\to F_{\mathscr{X}_1}^*
\sT_{\mathscr{X}_1/B_1}$ and $
\sT_{\mathscr{X}_1/B_1}\to \Omega_{dR}^*$ which induce the corresponding morphism on higher direct images:
$$
\Gamma: R^1f_{*,dR}(F_{\mathscr{X}_1}^*
\sT_{\mathscr{X}_1/B_1},\nabla_{can})\to R^1f_*F_{\mathscr{X}_1}^*\sT_{\mathscr{X}_1/B_1},
$$
and
$$
\Lambda: R^1f_*\sT_{\mathscr{X}_1/B_1}\to R^1f_{*,dR}(F_{\mathscr{X}_1}^*
\sT_{\mathscr{X}_1/B_1},\nabla_{can}).
$$
The coherent sheaves $R^1f_{*,dR}(F_{\mathscr{X}_1}^*
\sT_{\mathscr{X}_1/B_1},\nabla_{can})$ and $R^1f_*\sT_{\mathscr{X}_1/B_1}$ are both locally free, and of rank $3g-2+r$ and $3g-3+r$ respectively by Proposition \ref{dg_of_A} (2)-(3) below. Moreover, the sheaf morphisms $\Gamma$ and $\Gamma\circ \Lambda$ are injective as they are injective at each fiber over a closed point $\neq 0$. Define $\sB$ (resp. $\sW_F$) to be subsheaf $\Gamma(R^1f_{*,dR}(F_{\mathscr{X}_1}^*
\sT_{\mathscr{X}_1/B_1},\nabla_{can}))$ (resp. $\Gamma\circ \Lambda(R^1f_*\sT_{\mathscr{X}_1/B_1})$). Clearly, $\sW_F\subsetneq \sB$. Take and then fix a nowhere vanishing section $\xi_0$ of $\sB-\sW_F$ (it always exists after shrinking $B_1$ if necessary). We define $\sA\subset R^1f_*F_{\mathscr{X}_1}^*\sT_{\mathscr{X}_1/B_1}$ to be the translation of $\sW_F$ by $\xi_0$.

{\itshape Construction of $\sK$:} We shall follow the construction of the periodic cone in the smooth case, see Section \ref{consec}. In Step 1, one replaces $p_i$ with the natural maps in the following Cartesian diagram of the fiber product:
$$
\begin{CD}
\text{Jac}(\mathscr{X}_1/B_1)\times_{B_1}\mathscr{X}_1@>p_2>>\mathscr{X}_1\\
@Vp_1VV @VVfV\\
\text{Jac}(\mathscr{X}_1/B_1)@>>>B_1,
\end{CD}
$$
where $\text{Jac}(\mathscr{X}_1/B_1)$
is the relative Jacobian of the semistable curve $f: \mathscr{X}_1\to B_1$. Let $\sL$ be the universal line
bundle over $\text{Jac}(\mathscr{X}_1/B_1)\times_{B_1}\mathscr{X}_1$.
Consider the vector bundle
$$
\V:=p_{1*}\mathcal{H}om(p_2^*F^*\sT_{\mathscr{X}_1/B_1},p_2^{*}
\sT_{\mathscr{X}_1/B_1}^{\frac{p+1}{2}}\otimes \sL^{-1}),
$$
and its associated projective bundle $\P$. Consider the following commutative diagram of fiber products:
$$
\xymatrix{
 \V\times_{B_1} \mathscr{X}_1   \ar[r]^{\pi}    \ar[d] _{g_1}               & \text{Jac}(\mathscr{X}_1/B_1)\times_{B_1}\mathscr{X}_1 \ar[d]_{p_1}    \ar[r]^{ \qquad   p_2}&   \mathscr{X}_1\ar[d]^{f}\\
 \V \ar[r]^{\pi_0 }    &  \text{Jac}(\mathscr{X}_1/B_1)\ar[r]^{}   &  B_1. }
$$
Set $g_2$ (resp. $v$) to be the composite of the upper (resp. lower) horizontal arrows. Note that they form the fiber product
$$
\begin{CD}
\V\times_{B_1}\mathscr{X}_1@>g_2>>\mathscr{X}_1\\
@Vg_1VV @VVfV\\
\V@>v>>B_1.
\end{CD}
$$
Then Step 2 proceeds as in the smooth case, so one obtains a closed subset $\textrm{Ker}(\phi)\subset R^1g_{1,*}g_2^*F^*\sT_{\mathscr{X}_1/B_1}$. By the flat base change theorem, one has a sheaf isomorphism
$$
R^1g_{1,*}g_2^*F^*\sT_{\mathscr{X}_1/B_1}\cong v^*R^1f_*F^*\sT_{\mathscr{X}_1/B_1},
$$
while the latter as vector bundle is isomorphic to $R^1f_*F^*\sT_{\mathscr{X}_1/B_1}\times_{B_1}\V$. By a similar argument in Step 3, $\textrm{Ker}(\phi)\cap (R^1f_*F^*\sT_{\mathscr{X}_1/B_1}\times_{B_1}(\V-\{0\}))$ descends to a closed subset of $R^1f_*F^*\sT_{\mathscr{X}_1/B_1}\times_{B_1}\P$, which we denote again by $\textrm{Ker}(\phi)$. So we consider the proper map
$$
\pi_1: \textrm{Ker}(\phi)\to R^1f_*F^*\sT_{\mathscr{X}_1/B_1}
$$
which is the composite
$$
\textrm{Ker}(\phi)\hookrightarrow R^1g_{1,*}g_2^*F^*\sT_{\mathscr{X}_1/B_1}\cong R^1f_*F^*\sT_{\mathscr{X}_1/B_1}\times_{B_1}\P\to R^1f_*F^*\sT_{\mathscr{X}_1/B_1}.
$$
We define $\sK$ to be the closed subvariety $\pi_1(\textrm{Ker}(\phi))$.

By specializing to a closed point $b\neq 0\in B_1$, one finds that $\sA_{b}$ is a nonzero translation of $\sW_{F,b}$. Of course, this may differ from the original definition of $A$ in the smooth case by a translation. See Proposition \ref{affine}. But it is clear that $\sA_{b}\cap \sK_{b}\neq \emptyset$ iff $A\cap K\neq \emptyset$ by Remark \ref{twist}. The following paragraphs are devoted to $\sA_0\cap \sK_0$. For simplicity of notation, we shall again use $A$ for $\sA_0$, $B$ for $\sB_0$, $K$ for $\sK_0$, $\alpha$ for the induced morphism of $\Gamma$ on fibers over 0, and similarly $\beta$ for $\Lambda$ over 0.

\subsection{The $(0,3)$-curve}\label{subsection on 0,3 curve}
We study first the unique $(0,3)$-curve up to isomorphism. We equip $\chi=\P^1$ with the log structure associated to the divisor of three distinct points. In this case, $\omega_{\chi/k}=\sO_{\P^1}(1)$ and we are considering the following logarithmic Higgs bundle
$$
(\sO_{\P^1}(1)\oplus \sO_{\P^1},\theta), \theta=id: \sO_{\P^1}(1)\to \sO_{\P^1}\otimes \omega_{\chi/k}.
$$
Thus the ambient vector space $H^1(\chi, F^*\sT_{\chi/k})=H^1(\P^1,\sO_{\P^1}(-p))$ is of dimension $p-1$; The linear subspace
$$
W_F =\textrm{Im}(F^*: H^1(\chi, \sT_{\chi/k})\to H^1(\chi, F^*\sT_{\chi/k}))
$$
is of dimension zero. Hence $A$ is just a nonzero vector inside the line $B: \textrm{Im}(H^1_{dR}(F^*\sT_{\chi/k},\nabla_{can})\to H^1(\chi, F^*\sT_{\chi/k}))$. Here the dimension and the injectivity are proved in the same way as in Lemma \ref{dimension of B} and Lemma \ref{injective}.

\begin{lemma}\label{unique intersection in (0,3)-curve}
The periodic cone $K\subset H^1(\chi, F^*\sT_{\chi/k})$ is the whole space. Therefore, $A\cap K$ consists of one point $\xi_0$.
\end{lemma}
\begin{proof}
Take any $\xi\in H^1(\chi, F^*\sT_{\chi/k})$ which is the extension class of the following short exact sequence:
\begin{align}\label{eq625}
0\to\mathcal{O}_{\P^1} \to H_{\xi}  \to F^*\omega_{\chi/k}=\sO_{\P^1}(p) \to 0.
\end{align}
Note that $H_{\xi}\cong \mathcal{O}_{\P^1}(d_1)\oplus  \mathcal{O}_{\P^1}(d_2)$ for two integers $d_1,d_2$ with $d_1+d_2=p$. Thus the larger number between $d_1$ and $d_2$ must be $\geq\frac{p+1}{2}$. This implies that $H_{\xi}$ has an invertible subsheaf of degree $\frac{p+1}{2}$. Therefore, $K$ is just the whole space $H^1(\chi, F^*\sT_{\chi/k})$.  The unique element of ${\xi_0}=A\cap K$ corresponds to the following nonsplit extension (unique up to isomorphism):
\begin{align}\label{eq627}
  0\to(\mathcal{O}_{\P^1},\nabla_{can}) \to (H_0,\nabla)  \to (F^*\omega_{\chi/k},\nabla_{can}) \to 0.
\end{align}
\end{proof}
Now we proceed to a clearer description of elements in $K=H^1(\chi, F^*\sT_{\chi/k})$. Assume the three marked points of $\chi$ are given by $\{0,1,\infty\}$. Then there is an obvious $W_2$-lifting, say $\tilde \chi$, of the log curve $\chi$, which is given by $(\P^1,0,1,\infty)$ over $W_2$. Let $C=C_{\chi\subset \tilde \chi}$ (resp. $C^{-1}=C_{\chi\subset \tilde \chi}^{-1}$) be the Cartier transform (resp. inverse Cartier transform), see Appendix. It is an equivalence of categories. From an extension of flat bundles,
\begin{align}\label{extension at (0,3)-curve}
0\to (\sO_{\chi},\nabla_{can})\to (H,\nabla)\to (F^*\omega_{\chi/k},\nabla_{can})\to 0,
\end{align}
one obtains an extension of Higgs bundles
\begin{align}\label{extension of Higgs bunlde at (0,3)-curve}
0\to (\sO_{\chi},0)\to (E,\theta)=C(H,\nabla)\to (\omega_{\chi/k},0)\to 0,
\end{align}
and vice versa. Forgetting the Higgs field in the above extension, one sees that $E$ is the direct sum $\sO_{\chi}\oplus \omega_{\chi/k}$, and therefore $(E,\theta)$ is of the form
$$
(E,\theta_a):= (\omega_{\P^1_{\log}}\oplus \sO_{\P^1},\theta_a), \ \theta=a\cdot id: \omega_{\P^1_{\log}}\to \sO_{\chi}\otimes \omega_{\P^1_{\log}} =\omega_{\P^1_{\log}}\ \textrm{for some}\ a\in k.
$$
Here $\omega_{\P^1_{\log}}$ is the sheaf of log differentials on $\P^1$ with poles along $\{0,1,\infty\}$. For each $(E,\theta_a)$, there is a natural extension
\begin{align}\label{explicit extension of Higgs bunlde at (0,3)-curve}
0\to (\sO_{\chi},0)\stackrel{i}{\to} (E,\theta_a)\stackrel{p}{\to} (\omega_{\chi/k},0)\to 0,
\end{align}
where $i: \sO_{\chi}\to E=\omega_{\chi/k}\oplus \sO_{\chi}$ is the natural inclusion $(0,id)$ and $p: E\to\omega_{\chi/k}$ is the natural projection $\left(
                                                                                                                                                          \begin{array}{c}
                                                                                                                                                            id \\
                                                                                                                                                            0\\
                                                                                                                                                          \end{array}
                                                                                                                                                        \right)
$. Set $(H_a,\nabla_a):=C_{\chi\subset \tilde \chi}^{-1}(E,\theta_a)$. Applying the inverse Cartier transform to the extension  (\ref{explicit extension of Higgs bunlde at (0,3)-curve}), we obtain the following extension of $(H_a,\nabla_a)$:
\begin{align}\label{explicit extension at (0,3)-curve}
0\to (\sO_{\chi},\nabla_{can})\to (H_a,\nabla_a)\to (F^*\omega_{\chi/k},\nabla_{can})\to 0.
\end{align}
It is clear that the extension (\ref{explicit extension at (0,3)-curve}) is nonsplit iff $a\neq 0$.
\begin{lemma}\label{max degree subsheaf}
For $a\neq 0$, $H_a=  \mathcal{O}_{\P^1}(\frac{p+1}{2})\oplus  \mathcal{O}_{\P^1}(\frac{p-1}{2})$. In particular, $H_a, a\neq 0$ admits a unique invertible subsheaf $\omega_{\chi/k}^{\frac{p+1}{2}}= \mathcal{O}_{\P^1}(\frac{p+1}{2})$ of degree $\frac{p+1}{2}$.
\end{lemma}
\begin{proof}
The idea for the proof is similar to that of Proposition \ref{prop20}. Note it suffices to show the maximal destabilizing subsheaf of $H_a$ has degree $\frac{p+1}{2}$. First of all, we have already seen in the proof of Lemma \ref{unique intersection in (0,3)-curve} that degree $\frac{p+1}{2}$ invertible subsheaf exists. Let $L\subset H_a$ be any invertible subsheaf of positive degree. Then $L$ cannot be invariant under $\nabla_a$, since the $p$-curvature of $\nabla_a$ is nonzero. So we get a nonzero morphism
$$
\bar \nabla_a: L\to (H_a/L)\otimes \omega_{\chi/k},
$$
which implies $\deg L\leq \frac{\deg H_a+\deg \omega_{\chi/k}}{2}=\frac{p+1}{2}$ as claimed.
\end{proof}
In the above proof, for the maximal destabilizing subsheaf $\omega_{\chi/k}^{\frac{p+1}{2}}\subset H_a$, one obtains further an isomorphism:
\begin{align}\label{graded iso}
 \bar \nabla_a: \omega_{\chi/k}^{\frac{p+1}{2}}\stackrel{\cong}{\longrightarrow} (H_a/\omega_{\chi/k}^{\frac{p+1}{2}}\cong \omega_{\chi/k}^{\frac{p-1}{2}})\otimes \omega_{\chi/k}.
\end{align}

Consider further
\[\xymatrix{
0\ar[r] &   \sO_{\chi}   \ar[r] &    H_a   \ar[r]   &   F^*\omega_{\chi/k}  \ar[r] & 0\\
        &                        & \omega_{\chi/k}^{\frac{p+1}{2}} \ar_{}[u]\ar_{s}[ur] & & \\
}\]
Here $s: \omega_{\chi/k}^{\frac{p+1}{2}}\to F^*\omega_{\chi/k}$ is the composite $\omega_{\chi/k}^{\frac{p+1}{2}}\hookrightarrow H_a\twoheadrightarrow F^*\omega_{\chi/k}$.
\begin{proposition}\label{nonzero position of hodge filtration}
Assume $a\neq 0$ and let $P$ be any marked point of $\chi$. Then the value $s(P)$ is nonzero.
\end{proposition}
We shall take first an important digression into a description of residues of $\nabla_a$ at marked points, which will be applied several times in the later paragraphs. Let us look at the connection of around the node $0$. Let $[x:y]$ be the homogenous coordinate of $\P^1$ with $t=x/y$ an affine coordinate of $\{y\neq 0\}\subset \P^1$. Then $U=\Spec \ k[t,(t-1)^{-1}]$ is an open affine neighborhood of $0$. Take $\tilde U=\Spec \ W_2[t,(t-1)^{-1}]\subset \tilde \chi$, together with the standard log Frobenius lifting $\tilde F$ determined by $t\mapsto t^p$. Use the global basis $e_1=1$ for $\sO_{\P^1}$ and the local basis $e_2=\frac{dt}{t}$ for $\omega_{\P^1_{\log}}$ over $U$. Then by the construction of the inverse Cartier transform as in the Appendix, a local expression of the connection $\nabla_a$ over $U$ is given by:
$$
\nabla_a\{e_1\otimes 1, e_2\otimes 1\}=\{e_1\otimes 1, e_2\otimes 1\}\left(
                                                                 \begin{array}{cc}
                                                                   0 & a^p\frac{dt}{t} \\
                                                                   0 & 0 \\
                                                                 \end{array}
                                                               \right),
$$
where $[e_i]:=e_i\otimes 1, i=1,2$ is the \emph{natural} basis of $H$ over $U$. It follows immediately that the residue of $\nabla_a$ at the origin is expressed by
$$
\textrm{Res}_{0}(\nabla_a)\{[e_1](0),[e_2](0)\}=\{[e_1](0),[e_2](0)\}\left(
                                                                 \begin{array}{cc}
                                                                   0 & a^p \\
                                                                   0 & 0 \\
                                                                 \end{array}
                                                               \right).
$$
Using the transformation $t\mapsto t-1$ (resp. $t\mapsto t^{-1}$), one obtains the analogue for the marked points 1 (resp. $\infty$).

\begin{proof}[Proof of Proposition \ref{nonzero position of hodge filtration}]
Let $P$ be any marked point. By the previous description, the $k$-linear map $\textrm{Res}_{P}(\nabla_a): H_a(P)\to H_a(P)$ has the property that $\textrm{Res}_{P}(\nabla_a)([e_1](P))=0$. On the other hand, the isomorphism (\ref{graded iso}) induces the isomorphism at the point $P$:
$$
\textrm{Res}_{P}(\bar \nabla_a): \omega^{\frac{p+1}{2}}_{\chi/k}(P)\stackrel{\cong}{\longrightarrow} \omega^{\frac{p-1}{2}}_{\chi/k}(P).
$$
Now, if $s(P)$ were zero then $[e_1](P)$ would be a basis for the image of $\omega^{\frac{p+1}{2}}_{\chi/k}(P)\to H_a(P)$. But this is impossible.
\end{proof}
The last lemma of this subsection will be applied in the next subsection.
\begin{lemma}\label{inv under auto}
Let $\phi: \P^1\to \P^1$ be an automorphism preserving $\{0,1,\infty\}$ (i.e. an automorphism of the log curve $\chi$). For any $a\in k$, $\phi^*(H_a,\nabla_a)=(H_a,\nabla_a)$.
\end{lemma}
\begin{proof}
It is equivalent to show $\phi^*(E,\theta_a)=(E,\theta_a)$. But this is obvious.
\end{proof}

\subsection{A totally degenerate $(g,r)$-curve}
Let $\chi$ be  a t.d. $(g,r)$-curve. A direct calculation shows that $\chi$ has $\nu=2g-2+r$ irreducible components and $\delta=3g-3+r$ nodes. Let $\{P_i\in \chi|1\leq i \leq \delta\}$ be the set of nodes of $\chi$. Let $\gamma:\tilde \chi\to \chi$ be the normalization and $\{\chi_i| 1\leq i \leq \nu\}$ its irreducible components. Considering the inverse image of nodes as marked points on $\tilde \chi$, $\tilde \chi$ is therefore a disjoint union of $(0,3)$-curves. In other words, one obtains $\chi$ by gluing $\nu$-copies of $(0,3)$-curves along the marked points in a suitable way. We can assume each copy $\chi_i$ to be $(\P^1, 0,1, \infty)$.
\begin{proposition}\label{dg_of_A}
\begin{enumerate}
\item $\dim H^1(\chi, F^*\sT_{\chi/k})= (2p+1)(g-1)+pr$;
\item The map $F^*=\alpha\circ \beta: H^1(\chi,\sT_{\chi/k})\to H^1(\chi,F^*\sT_{\chi/k})$ is injective and thus $\dim W_F=3g-3+r$;
\item The map $\alpha: H^1_{dR}(F^*\sT_{\chi/k},\nabla_{can})\to H^1(\chi,F^*\sT_{\chi/k})$ is injective and hence $\dim B=3g-2+r$.
\end{enumerate}
\end{proposition}
\begin{proof}
(1) One has the following short exact sequence of $\sO_{\chi}$-modules:
\begin{align}\label{eq605}
0 \to F^*\sT_{\chi/k} \to \gamma_*\gamma^*(F^*\sT_{\chi/k})\to \oplus_{i=1}^\delta
k_{P_i}\to 0,
\end{align}
where $k_{P_i}$ denotes the skyscraper sheaf with stalk $k$ at $P_i$ and $0$ elsewhere. Over each component $\chi_i$, one has
$$
\gamma^*(F^*\sT_{\chi/k})|_{\chi_i}\cong F^*\sT_{\chi_i/k}\cong \sO_{\P^1}(-p).
$$
Indeed, one verifies that the natural morphism of locally free $\sO_{\chi_i}$-sheaves $$\gamma^*(\omega_{\chi/k})|_{\chi_i}\to \omega_{\chi_i/k}$$ is an isomorphism (which is \emph{not} automatic as pointed out by the referee since $\gamma$ is only a morphism of schemes). Since $H^1(\chi,\gamma_*\gamma^*(F^*\sT_{\chi/k}))=H^1(\widetilde{\chi}, \gamma^*(F^*\sT_{\chi/k}))$ as
$\gamma$ is finite, it follows that
\[h^1(\chi, F^*\sT_{\chi/k})=\sum_{i=1}^\nu h^1(\chi_i,\sO_{\P^1}(-p))+\delta= (2p+1)(g-1)+pr. \]

(2) Using the long exact sequence of cohomologies associated with the short exact sequence
\begin{align}\label{eq613}
0 \to \sT_{\chi/k} \to \gamma_*\gamma^* \sT_{\chi/k}\to \oplus_{i=1}^\delta
k_{P_i}\to 0,
\end{align}
one calculates the dimension:
$$h^1(\chi, \sT_{\chi/k})=h^0(\chi,  \oplus_{i=1}^\delta k_{P_i})=\delta.$$
It remains to show the injectivity of $F^*$. Like in the smooth case, we shall show that if an extension of the form
\begin{align}\label{eq610}
0\to  \sO_{\chi}   \to  E \to  \omega_{\chi/k} \to 0
\end{align}
is nonsplit, then its Frobenius pull-back is nonsplit.

The restriction of $E$ to each component $\chi_i$ is split. In other words, $E$ is obtained by gluing copies of $ \omega_{\P^1_{\log}}\oplus \sO_{\P^1}$ along marked points.  Any node $P_i$ of $\chi$ is obtained either by gluing two different components along one of three marked points or by gluing one component along two different marked points (i.e. the case of self normal crossing). Since the bundle $\omega_{\P^1_{\log}}\oplus \sO_{\P^1}$ does not change under an automorphism of the log curve $\P^1$, we can assume that \emph{formal} locally at $P_i$ the bundle $E$ is described by gluing two copies of  $[\omega_{\P^1_{\log}}\oplus \sO_{\P^1}]|_{
\hat{U}}$ at 0, where $\hat U$ is the completion at $0$ of the open affine $U$ considered in the $(0,3)$-curve case, and $[\omega_{\P^1_{\log}}\oplus \sO_{\P^1}]|_{
\hat{U}}$ is the pull-back of $[\omega_{\P^1_{\log}}\oplus \sO_{\P^1}]|_{U}$ via the base change $\hat U\to U$. We thank the referee for suggesting us to work formal locally in the case of self-normal crossing. Thus, in the case of self-normal crossing, our computations below should be understood over $\hat U$ by taking the obvious base change. Take the basis $\{e_1\}$ (resp. $\{e_2\}$) of $\sO_{\P^1}$ (resp. $\omega_{\P^1_{\log}}$) in the open affine neighborhood $U$ of $0$. Let $\{e_{1,i},e_{2,i}\}_{i=1,2}$ be two copies of such. We shall see that the set of isomorphism classes of extensions (\ref{eq610}) is in one-to-one correspondence with the set $\{A_i\}_{1\leq i\leq \delta}$ with $A_i$ the following upper triangular matrix
$$
A_i=\left(
                       \begin{array}{cc}
                         1 & \lambda_i \\
                        0& -1 \\
                       \end{array}
                     \right).
$$
Indeed, the formula
\begin{align}\label{eq612}
\{e_{1,2}(0),e_{2,2}(0)\}=\{e_{1,1}(0),e_{2,1}(0)\}A_i
\end{align}
gives the gluing data of the fibers of two copies of $\sO_{\P^1}\oplus \omega_{\P^1_{\log}}$ at 0 and any different $\lambda_i$ gives a locally non-isomorphic $E$ over an open neighborhood of $P_i$. The second diagonal element in $A_i$ is $-1$ for the following reason: Let $\{t_i\}_{i=1,2}$ be the local coordinate of an open affine neighborhood $U_i\subset \chi_i, i=1,2$ of the node $P_i$. Set $U:=U_1\cup U_2$. Then $\omega_{\chi/k}(U)$ is expressed by
$$
\frac{\sO_U\{\frac{dt_1}{t_1}\}\oplus \sO_{U}\{\frac{dt_2}{t_2}\}}{\sO_U\{\frac{dt_1}{t_1}+\frac{dt_2}{t_2}\}}.
$$
Therefore, $\frac{dt_2}{t_2}=-\frac{dt_1}{t_1}$ is the transition relation for the gluing of $\omega_{\chi/k}$ at the node. Thus, we have obtained an explicit $k$-linear isomorphism
$$
H^1(\chi, \sT_{\chi/k})\cong k^{\delta}, \quad [E]\mapsto (\lambda_1,\cdots,\lambda_{\delta}).
$$

Now let $\{[e_{1,i}], [e_{2,i}]\}_{i=1,2}$ be the natural basis of the restrictions of $F^*E$ to $\chi_1$ and $\chi_2$. It is clear that at $P_i$, the following formula holds:
$$
\{[e_{1,2}](0), [e_{2,2}](0)\}=\{[e_{1,1}](0), [e_{2,1}](0)\} \left(
                       \begin{array}{cc}
                         1 & \lambda_i^p \\
                        0& -1 \\
                       \end{array}
                     \right).
$$
Thus, if $F^*E$, as an extension of $F^*\omega_{\chi/k}$ by $F^*\sO_{\chi}$, is nonsplit, then there is a nonzero $\lambda_{i_0}^p$ for some $i_0$. As $\lambda_{i_0}$ is nonzero, $E$ is nonsplit as well. \\

(3) Note that one still has the isomorphism
$$
H^1_{dR}(F^*\sT_{\chi/k},\nabla_{can})\cong \textrm{Ext}^1((F^*\omega_{\chi/k},\nabla_{can}),(F^*\sO_{\chi},\nabla_{can})).
$$
The latter space classifies the isomorphism classes of extensions of the following form:
\begin{align}\label{eq604}
0\to  (F^*\sO_{\chi},\nabla_{can})   \to  (H, \nabla) \to (F^*\omega_{\chi/k}, \nabla_{can}) \to 0.
\end{align}
Like in the smooth case, one interprets $\alpha$ as the forgetful map
$$
\textrm{Ext}^1((F^*\omega_{\chi/k},\nabla_{can}),(F^*\sO_{\chi},\nabla_{can}))\to \textrm{Ext}^1(F^*\omega_{\chi/k},F^*\sO_{\chi})\cong H^1(\chi, F^*\sT_{\chi/k}).
$$
Next we shall show that if the bundle $H$ of the extension $(H,\nabla)$ of form (\ref{eq604}) is split, then the extension is split. The proof is similar to (2) but easier. First of all, the splitting of $H$ restricts to a splitting of $H|_{\chi_i}$ on the irreducible component $\chi_i$ for each $i$. As $\chi_i$ is the $(0,3)$ curve, it follows that $(H,\nabla)|_{\chi_i}$ is also split. Thus $\nabla|_{\chi_i}$ is nothing but the canonical connection $\nabla_{can}$. Therefore, $\nabla$ is just the canonical connection on $H=F^*\omega_{\chi/k}\oplus F^*\sO_{\chi}$, and therefore $(H,\nabla)$ is split.

As $H^1_{dR}(F^*\sT_{\chi/k},\nabla_{can})$ is the fiber of the coherent sheaf $R^1f_{*,dR}(F_{\mathscr{X}_1}^*
\sT_{\mathscr{X}_1/B_1},\nabla_{can})$ at $0\in B_1$, its dimension is greater than or equal to $3g-2+r$ by upper semi-continuity. Thus we need only show $\dim H^1_{dR}(F^*\sT_{\chi/k},\nabla_{can})\leq 3g-2+r$. Consider the following commutative diagram of linear morphisms:
$$
\xymatrix{
 B= H^1_{dR}(F^*\sT_{\chi/k},\nabla_{can}) \ar[d]_{\alpha} \ar[r]^{{\rm res}|_{\chi_1}} & H^1_{dR}(F^*\sT_{\chi_1/k},\nabla_{can}) \ar[d]^{\alpha} \\
 H^1(\chi,F^*\sT_{\chi/k}) \ar[r]^{{\rm res}|_{\chi_1}} & H^1(\chi_1,F^*\sT_{\chi_1/k}).}
$$
The vertical morphisms $\alpha$ are injective by the previous discussion. Set $L=\alpha(H^1_{dR}(F^*\sT_{\chi_1/k},\nabla_{can}))$ (which is just the line $B$ in the subsection \ref{subsection on 0,3 curve}) and let $B_0$ be the kernel of the morphism ${\rm res}|_{\chi_1}: \alpha(B)\to L$.  We shall get a key characterization of the image $\alpha(B)\subset  H^1(\chi,F^*\sT_{\chi/k})$ by the following calculation involving residues:

Take an extension of the form (\ref{eq604}). Let $(H_i,\nabla_i)$ be the restriction of $(H,\nabla)$ to $\chi_i$, and let $P$ be a node obtained by gluing $\chi_1$ and $\chi_2$ (they may equal) along marked points. As discussed in the $(0,3)$-curve case, $(H_i,\nabla_i)=C^{-1}(E:=\omega_{\P^1_{\log}}\oplus \sO_{\P^1},\theta_{a_i})$ for a scalar $a_i\in k$. Also, the scalar $a_i$ is kept under any automorphism of the log curve by Lemma \ref{inv under auto}. Thus, we can assume that under the natural basis $\{[e_{1,i}](P),[e_{2,i}](P)\}$ of $H_i(P), i=1,2$, the gluing matrix at the node $P$ is given by
$\left(
                                                                           \begin{array}{cc}
                                                                             1 &
                                                                             \lambda \\
                                                                            0 &-1 \\
                                                                           \end{array}
                                                                         \right)
$,
that is, the linear transition relation from $H_1(P)$ to $H_2(P)$ is determined by
\begin{align}\label{transition relation B-matrix}
\{[e_{1,2}](P),[e_{2,2}](P)\}=\{[e_{1,1}](P),[e_{2,1}](P)\}\left(
                                                                           \begin{array}{cc}
                                                                             1 &
                                                                             \lambda \\
                                                                            0 &-1 \\
                                                                           \end{array}
                                                                         \right).
\end{align}
By the discussion on the residues of connections at marked points in the $(0,3)$-curve case, and since $\nabla_1$ and $\nabla_2$ glue at the node $P$, we obtain the equality
$$
\left(
  \begin{array}{cc}
    0 & a_1^p \\
    0 & 0 \\
  \end{array}
\right)\left(
         \begin{array}{cc}
           1 & \lambda \\
           0 & -1 \\
         \end{array}
       \right)=\left(
  \begin{array}{cc}
    1 & \lambda \\
    0 & -1 \\
  \end{array}
\right)\left(
         \begin{array}{cc}
           0& -a_2^p \\
           0 & 0 \\
         \end{array}
       \right).
$$
It follows that $a_1=a_2$. As $\chi$ is connected, it follows that for each $i$,
$$
(H_i,\nabla_i)=C^{-1}(E,\theta_{a})
$$
for an $a\in k$.

Assume now that the restriction to $\chi_1$ of an extension (\ref{eq604}) is split. Then $a=a_1=0$ and consequently, all restrictions to $\chi_i$s are split. This means that elements in $B_0$ are contained in those obtained by gluing copies of the split extension
$$
0\to F^*\sO_{\P^1}\to F^*\omega_{\P^1_{\log}}\oplus F^*\sO_{\P^1}\to F^*\omega_{\P^1_{\log}}\to 0
$$
along the nodes, which are exactly elements of the subspace $W_F\subset  H^1(\chi,F^*\sT_{\chi/k})$ in (2). Thus we have obtained the required inequality
$$
\dim B=\dim \alpha(B)\leq \dim B_0+\dim L\leq \dim W_F+1=3g-2+r.
$$
\end{proof}
The above proof shows that the morphism ${\rm res|_{\chi_1}}: \alpha(B)\to L$ is surjective and its kernel is exactly $W_F$.  That is, we have a short exact sequence of $k$-vector spaces
$$
0\to W_F\to \alpha(B)\stackrel{{\rm res|_{\chi_1}}}{\to}L\to 0.
$$
By equipping extensions in $W_F$ with the canonical connections,  one identifies naturally $W_F$ with its inverse image $\alpha^{-1}(W_F)\subset B$. On the other hand, recall that $A$ is by definition a translation of $W_F$ by the element $\xi_0$ in $B-W_F$.  Thus we know that $\xi_0$ restricts to $\chi_1$ (and also any other component $\chi_i$) a nonsplit extension. Thus the line $k\{\xi_0\}\subset B$ induces a splitting of the short exact sequence of $k$-vector spaces
$$
0\to W_F\to B\stackrel{{\rm res|_{\chi_1}}}{\to}H^1_{dR}(F^*\sT_{\chi_1/k},\nabla_{can})\to 0.
$$
Therefore, $B=W_F\oplus k\{\xi_0\}$.  After these preparations, we proceed now to prove $A\cap K\neq \emptyset$.
\begin{proposition}\label{AKinter}
For every totally degenerate $(g, r)$-curve, the set $A\cap K$ consists of one unique element.
\end{proposition}
\begin{proof}
Let $(H_0,\nabla)$ be the corresponding extension to $\xi_0$. Then the elements $(H,\nabla)$ of $A$ are in one-to-one correspondence with tuples $\mu:=\{\mu_1,\cdots, \mu_{\delta}\}\in k^{\delta}$, so that $\mu=0$ corresponds to $(H_0,\nabla)$. Explicitly, let $(H=H_{\mu},\nabla)$ be the corresponding flat bundle over $\chi$ and let $(H_i,\nabla_i)$ be the restriction of $(H,\nabla)$ to $\chi_i$, then the gluing datum under the natural basis at a node $P_i$ is given by a upper triangular matrix $B_i=\left(
                                                                           \begin{array}{cc}
                                                                             1 &
                                                                             \mu_i \\
                                                                            0 &-1 \\
                                                                           \end{array}
                                                                         \right)
$ with the transition relation
\begin{align}\label{transition relation B-matrix}
\{[e_{1,2}](P_i),[e_{2,2}](P_i)\}=\{[e_{1,1}](P_i),[e_{2,1}](P_i)\}B_i.
\end{align}
We claim that there exists a unique tuple $\mu$ such that the corresponding $H_{\mu}$ admits an invertible subsheaf $(\omega_{\chi/k})^{\frac{p+1}{2}}\otimes \mathscr{L}$ for some degree zero invertible sheaf $\mathscr{L}$. This claim is equivalent to saying that $A$ intersects with $K$ at one unique point (multiplicity not counted), and it involves again a gluing argument.

By Lemma \ref{max degree subsheaf}, we know that the restriction $H_i$ of $H$ to $\chi_i$ ($i$ arbitrary) admits a unique $\omega_{\chi_i/k}^{\frac{p+1}{2}}\hookrightarrow H_i$. Assume that $P_i$ is a node obtained by gluing $\chi_j,j=1,2$ (they may equal) along marked points. By Proposition \ref{nonzero position of hodge filtration}, the image of the fiber $\omega_{\chi_j/k}^{\frac{p+1}{2}}(P_i)$ inside $H_j(P_i)$ admits a basis of the form
$$
h_j=b[e_{1,j}](P_i)+[e_{2,j}](P_i).
$$
Thus, in order to glue these two invertible subsheaves $\omega_{\chi_j/k}^{\frac{p+1}{2}}, j=1,2$ along $P_i$ together, the necessary and sufficient condition is the one-dimensional subspace $\omega_{\chi_1/k}^{\frac{p+1}{2}}(P_i)\subset H_1(P_i)$ maps onto the one-dimensional subspace $\omega_{\chi_2/k}^{\frac{p+1}{2}}(P_i)\subset H_2(P_i)$ under the transition from $H_1(P_i)$ to $H_2(P_i)$, which is given by the matrix $B_i$ under the natural basis as above. By the relation (\ref{transition relation B-matrix}), it is equivalent to the solvability of the following equation:
$$
\left(
  \begin{array}{c}
    b \\
    1 \\
  \end{array}
\right)=u\left(
                  \begin{array}{cc}
                    1 & \mu_i \\
                    0 & -1 \\
                  \end{array}
                \right)\left(
           \begin{array}{c}
             b \\
             1 \\
           \end{array}
         \right), \quad\textrm{for some}\ u\in k.
$$
One obtains the unique solution $u=-1$ and $\mu_i=-2b$. It follows that there is a unique $\mu\in k^{\delta}$ such that the invertible subsheaves $\{\omega_{\chi_i/k}^{\frac{p+1}{2}}\}_i$ of $H_i$ glue into an invertible subsheaf of the corresponding $H_{\mu}$. Note that the sheaf
$\omega_{\chi/k}^{\frac{p+1}{2}}$ restricts to the sheaf $\omega_{\chi_{i}/k}^{\frac{p+1}{2}}$ over the component $\chi_i$ and the transition at each node is the multiplication by $(-1)^{\frac{p+1}{2}}$. The above calculation shows that the transition at each node of the glued sheaf is always the multiplication by $-1$. Thus, if $\frac{p+1}{2}$ is odd, they are indeed the same sheaf. If $\frac{p+1}{2}$ is even, they differ from each other by the invertible sheaf $ \mathscr{L}$ obtained by gluing $\sO_{\chi_i}$s along nodes with transitions by $-1$. However, $\mathscr{L}$ is two-torsion in either case and hence degree zero.
\end{proof}

Let $\xi_1$ be the unique point in $A\cap K$. Assume that
$\xi_1\in\mathrm{Ker}(\phi_s)$ for some section  $s\in
\Hom((\omega_{\chi/k})^{\frac{p+1}{2}}\otimes \mathscr{L}, F^*(\omega_{\chi/k}) )$, where $\mathscr{L}$ is a certain two-torsion line bundle. Its dual $\check s$ induces the composite
\begin{align}\label{eq614}
 F^*\sT_{\chi/k}=\sT_{\chi/k}^{p} \overset{\check{s}}\to  \sT_{\chi/k}^{\frac{p+1}{2}}\otimes \mathscr{L}^{-1} \overset{\check{s}}\to  \sT_{\chi/k}\otimes \mathscr{L}^{-2}= \sT_{\chi/k}.
\end{align}
The induced map on the first cohomologies is also denoted by $\check{s}^2$.
\begin{proposition}\label{tdord}
Use notation as above. The composite map
\begin{align} \label{ordinary3}
H^1(\chi, \sT_{\chi/k}  )\overset{F^*}\to H^1(\chi,  F^*\sT_{\chi/k})
\overset{\check{s}^2}\to H^1(\chi, \sT_{\chi/k}  )
\end{align}
is injective.
\end{proposition}
\begin{proof}
First of all, Proposition \ref{nonzero position of hodge filtration} shows that the value of $s$ at each node of $\chi$ is nonzero. So is its dual $\check s$. This fact gives rise to the isomorphism in the lower right corner of the
following commutative diagram:
\[\xymatrix{
0   \ar[r] &    \sT_{\chi/k}  \ar[r] \ar[d]^{F^*}  &\gamma_*\gamma^* \sT_{ \chi/k}    \ar[d]^{F^*}   \ar[r]& \oplus_{i=1}^\delta
k_{P_i}   \ar[r]  \ar[d]^{(\cdot)^p} &  0\\
0  \ar[r] &F^*\sT_{\chi/k}   \ar[r]   \ar[d]^{\check{s}^2} & \gamma_*\gamma^* F^* \sT_{ \chi/k}   \ar[d]^{\check{s}^2}   \ar[r] & \oplus_{i=1}^\delta
k_{P_i} \ar[d]^{\cong}  \ar[r]  &  0\\
0   \ar[r] &    \sT_{\chi/k}  \ar[r]  &\gamma_*\gamma^* \sT_{ \chi/k}
\ar[r]&  \oplus_{i=1}^\delta
k_{P_i}     \ar[r] &  0. }\]
After taking cohomologies, we get
\[\xymatrix{
 \oplus_{i=1}^\delta
k_{P_i}   \ar[rr]^-{\cong}  \ar[d]^{(\cdot)^p} & &  H^1( \sT_{\chi/k})  \ar[d]^{F^*}  \\
   \oplus_{i=1}^\delta
k_{P_i}\ar[d]^{\cong}  \ar[rr]^-{\hookrightarrow} & & H^1(F^*\sT_{\chi/k})     \ar[d]^{\check{s}^2}   \\
  \oplus_{i=1}^\delta
k_{P_i}  \ar[rr]^-{\cong } & &   H^1(\sT_{\chi/k})
}\]
Clearly (\ref{ordinary3}) is injective. This completes the proof.
\end{proof}

\subsection{Nonemptiness of $A\cap K$ and ordinariness for a generic smooth $(g,r)$-curve}\label{relative case}
Building upon Proposition \ref{AKinter} and Proposition \ref{tdord}, we show the nonemptiness of $A\cap K$ and the existence of ordinary sections for a generic smooth $(g,r)$-curve, by applying a continuity argument.
\begin{proposition}\label{generic nonempty intersection}
For a generic curve $X_1$ in $\mathscr{M}_{g,r}$, $A$ and $K$ intersect.
\end{proposition}
\begin{proof}
Take an arbitrary t.d. curve $\chi\in \overline{\mathscr{M}}_{g,r}$ and a log smooth deformation $f: \mathscr{X}_1\to B_1$ of $\chi$ as given in \S\ref{section on relative A and K}. Write $\sV$ for the vector bundle $R^1f_*(F^*\sT_{\mathscr{X}_1/B_1})$ over $B_1$. Consider the irreducible closed subsets $\sA$ and  $\sK$ and the composite
$$
\pi: \sA\cap \sK\subset \sV\to B_1.
$$
We know that: (i). $\dim \sA+\dim \sK=\dim \sV+\dim B_1$ which can be seen from Formula (\ref{eq07}). This equality implies that any irreducible component of $\sA\cap \sK$ is of dimension $\geq \dim \sA+\dim \sK-\dim \sV=\dim B_1=1$; (ii). By Proposition \ref{AKinter}, $\xi_1= \pi^{-1}(0)$. Take any irreducible component $Z$ of $\sA\cap \sK$ passing through $\xi_1$ and consider the restriction of $\pi$ to $Z$: $\pi_{Z}: Z\to B_1$. Since the fiber of $\pi_Z$ over $0$ is zero dimensional, the morphism $\pi_Z$ has to be \emph{dominant}, and moreover, its relative dimension is zero (i.e. $Z$ is a \emph{curve}). Thus, there is a nonempty open subset of $B_1$ such that the fiber of $\pi_Z$ over a closed point in which consists of finitely many points. Shrinking this open subset if necessary, we shall obtain an open subset $V$ of $B_1$ such that the fiber of $\sA\cap \sK$ over a closed point in which has an isolated point as irreducible component.\\

Let $0\neq b_1\in B_1$ be a closed point in $V$, and let $X_1$ be the fiber of $f$ over $b_1$, which is a smooth $(g,r)$-curve. Let $B_2$ be an \emph{open} affine neighborhood of $b_1$ in $\mathscr{M}_{g,r}$ (not the compactification!) such that a local universal family $g: \mathscr{X}_2\to B_2$ exists. For $g$, we repeat the constructions of two closed subsets $\sA$ and $\sK$ inside the vector bundle $\sV:=R^1g_*F^*_{\mathscr{X}_2}\sT_{\mathscr{X}_2/B_2}$ as given in \S4.1. Take an isolated point $\xi_2$ in $A\cap K$, which is the fiber of $\sA\cap \sK$ over $b_1$, and an irreducible component $W$ of $\sA\cap \sK$ passing through $\xi_2$. Consider the composite of natural morphisms
$$
\pi_{W}: W\subset \sA\cap \sK\subset \sV\to B_2,
$$
which are defined similarly as above. Again, by Formula (\ref{eq07}), it follows that $\dim W\geq \dim \sA+\dim \sK-\dim \sV=\dim B_2$. Since the fiber of the map $\pi_{W}$ over $b_1$ is zero dimensional, one concludes again that $\pi_{W}$ is \emph{dominant} (and there is a nonempty open subset of $B_2$ over which closed fibers of $\pi_W$ are of zero dimensional). Therefore, there exists a nonempty Zariski open subset $U$ of $\mathscr{M}_{g,r}$ such that $A\cap K\neq \emptyset$ over any closed point in $U$. This completes the proof.
\end{proof}

We are now ready to supply a proof for Theorem \ref{thm03}.
\begin{proof}[Proof of Theorem \ref{thm03}]
This theorem follows from Proposition \ref{prop20} and Proposition \ref{generic nonempty intersection}.
\end{proof}

Finally, we proceed to show the existence of ordinary sections for a generic curve. First, we make the following

\begin{definition}\label{ordinary curve}
A smooth $(g,r)$-curve $X_1$ is called ordinary if $A\cap K\neq \emptyset$ and moreover if there exists an ordinary section $s
\in \text{Hom}(L_1\otimes \mathscr{L}, F^*L_1)$ such that  $A\cap
K \cap \mathrm{Ker}(\phi_s) \neq \emptyset$.
\end{definition}

\begin{proposition}\label{ordinary}
A generic curve $X_1$ in $\mathscr{M}_{g,r}$ is
ordinary.
\end{proposition}
\begin{proof}
By definition of ordinariness, the set of ordinary sections makes an open subset of $\text{Hom}(L_1\otimes \mathscr{L}, F^*L_1)$. Thus for every smooth $(g,r)$-curve, the subset $K^0:=\cup_{s \text{
ordinary}}\mathrm{Ker}(\phi_s)$ is an open subset of $K$ in view of
\eqref{eq05}. By extending to the relative case, we obtain an open
subscheme $\sK^0\subset \sK$. By Proposition \ref{tdord}, $\sA\cap \sK^0$ is a \emph{nonempty} open subset of $\sA\cap \sK$. Thus replacing $\sA\cap \sK$ with $\sA\cap \sK^0$ in the argument of Proposition \ref{generic nonempty intersection} gives the desired result.
\end{proof}

\section{Lifting periodic Higgs bundles to a truncated Witt ring}\label{lifting section}
In the previous two sections, we have basically established the existence of a periodic Higgs-de Rham flow initializing a Higgs bundle with maximal Higgs field (\ref{Higgsbdl}) in characteristic $p$. The purpose of this section is therefore to lift a given periodic Higgs-de Rham flow in characteristic $p$ to the mixed characteristic. This will be done inductively. Namely, assuming the existence of a periodic Higgs-de Rham flow over the Witt ring $W_{n-1}$ of length $n-1$, we show it can be lifted to a periodic Higgs-de Rham flow over $W_n$, under the ordinary condition.

Before moving on to the exact arguments via detailed calculations, we shall give an overview of the whole lifting process. To start with, we shall fix an ordinary curve $X_1$, together with an ordinary section $s\in \Hom(L_1\otimes \tL, F^*L_1)$ satisfying $A\cap K\cap {\rm Ker}(\phi_s)\neq \emptyset$ (see Definitions \ref{ordinary curve} and \ref{ordinary section}). Thus we get a morphism $\tilde{s}: L_1\otimes \tL\to H$, whose image is denoted by $Fil_1\subset H$, such that its composite with the natural projection $H\to F^*L_1$ is $s$. As shown in Proposition \ref{prop20}, we get a $W_2$-lifting $X_2$ of the log curve $X_1$ such that we can take an isomorphism (which is unique up to scalar in $k^*$):
$$
\phi_1: Gr_{Fil_1}\circ C^{-1}_{X_1\subset X_2}(E_1,\theta_1)\cong (E_1,\theta_1)\otimes (\tL,0).
$$
Moreover, this naturally makes the Higgs bundle $(E_1,\theta_1)$ two-periodic (see Corollary \ref{cor of thm 03}). Set $(H_1,\nabla_1)=C^{-1}_{X_1\subset X_2}(E_1,\theta_1)$ and let $(E_1,\theta_1,Fil_1,\phi_1)$ denote the two-periodic Higgs-de Rham flow over $X_1/k$ defined as above.\\

Let $(X_n,D_n)/W_n$ be a lifting of the log curve $(X_1,D_1)/k$, and let $\omega_{X_n/W_n}:=\Omega_{X_n}(\log D_n)$ be the sheaf of log differential forms with poles along $D_n$. The following simple lemma will be used later.
\begin{lemma}\label{unique lifting of maximal Higgs field}
Use notation as above. There exists a unique graded Higgs module $(E_n,\theta_n)$ over $X_n$ lifting $(E_1,\theta_1)$ up to isomorphism.
\end{lemma}
\begin{proof}
First of all, there exists a unique lifting $L_n$ of $L_1$ over $X_n$ such that $L_n^{\otimes 2}=\omega_{X_n/W_n}$. Let us show the uniqueness of lifting first, and let $L_n'$ be another such lifting. Then $L_n'=L_n\otimes M_n$ with $M_n$ an invertible $\sO_{X_n}$-module satisfying $M_n^{\otimes 2}=\sO_{X_n}$ and $M_n\otimes k=\sO_{X_1}$. As $p\geq 3$ by our assumption on the characteristic of $k$, $M_n$ defines a finite \'{e}tale covering of $X_n$ over $W_n$ of degree at most two which becomes trivial by reduction modulo $p$. Thus the covering is indeed trivial and $M_n=\sO_{X_n}$, which implies that $L_n'=L_n$. To show the existence, we proceed by induction on $n$. Let $\{\omega_{ij}\}$ be a \v{C}ech-representative of $\omega_{X_n/W_n}$. Modulo $p^{n-1}$, one gets the induced \v{C}ech-representative $\{\bar{\omega}_{ij}\}$ of $\omega_{X_{n-1}/W_{n-1}}$, for which we can assume the equation $l_{ij}^2=\bar{\omega}_{ij}$, where $\{l_{ij}\}$ is a \v{C}ech-representative of the unique lifting $L_{n-1}$ by induction hypothesis. Take an arbitrary lifting $\tilde{L}_{n}$ of $L_{n-1}$ with a \v{C}ech-representative $\tilde{l}_{ij}$ which lifts $l_{ij}$. We are going to solve $\{a_{ij}\}$ satisfying the equation
$$
(\tilde{l}_{ij}+p^{n-1}a_{ij})^2=\omega_{ij}.
$$
Clearly, it leads to the unique solution $a_{ij}=(2\bar{l}_{ij})^{-1}b_{ij}$, where $\bar{l}_{ij}$ is the mod $p$-reduction of $l_{ij}$ and $b_{ij}=\frac{\omega_{ij}-\tilde{l}_{ij}^2}{p^{n-1}}$. One checks furthermore that, by the 2-cocycle condition of $\{\omega_{ij}\}$, $\{\tilde{l}_{ij}+p^{n-1}a_{ij}\}$ satisfies the 2-cocycle condition, which defines a line bundle $L_n$ lifting $L_{n-1}$ and satisfying $L_n^{\otimes 2}=\omega_{X_n/W_n}$ as required. Therefore, we obtain a graded Higgs module $(E_n=L_n\oplus L_n^{-1},\theta_n)$ with
$$
\theta_n: L_n\cong L_n^{-1}\otimes \omega_{X_n/W_n},
$$
whose reduction mod $p$ gives
$$
\bar{\theta}_n: L_1\cong L_1^{-1}\otimes \omega_{X_1/k}.
$$
Note $\bar{\theta}_n$ differs from $\theta_1$ only by a nonzero scalar in $k$. Thus, by adjusting $\theta_n$ with a suitable scalar in $W_{n}$,
we can always require $(E_n,\theta_n)$ lift $(E_1,\theta_1)$. The uniqueness of $(E_n,\theta_n)$ lifting $(E_1,\theta_1)$ is also clear: since $L_n$ is unique, two liftings can only differ on the Higgs field. Fixing the unique lifting $\theta_{n-1}$ (up to isomorphism), the set of liftings $\theta_n$ forms a torsor under $H^0(X_1,\sO_{X_1})=k$ which means that two liftings differ by a scalar of form $\mu:=1+p^{n-1}\lambda\in W_n$. It is clear that $(E_n,\theta_n)$ and $(E_n,\mu\theta_n)$ are isomorphic graded Higgs modules.
\end{proof}

Now we address the problem of lifting the two-periodic Higgs-de Rham flow from $X_1/k$ to $X_2/W_2$. First, there is a unique lifting of the graded Higgs bundle $(E_2,\theta_2)$ over $X_2$ of $(E_1,\theta_1)$ by Lemma \ref{unique lifting of maximal Higgs field}. Second, by choosing \emph{any} $W_3$-lifting $X_3$ of the log curve $X_2/W_2$, we get the inverse Cartier transform
\footnote{For a sketch of the construction of the inverse Cartier transform over a truncated Witt ring, including the log curve case, we refer the reader to paragraphs after the proof of Theorem \ref{thm04} and before the proof of Proposition \ref{anti-obstr}.}
$C^{-1}_{X_2\subset X_3}(E_2,\theta_2)$ (see \S5 \cite{LSZ}) which is a flat bundle over $X_2$ whose reduction mod $p$ is $(H_1,\nabla_1)$. Third, given $(H_2,\nabla_2):=C^{-1}_{X_2\subset X_3}(E_2,\theta_2)$ as above, there is an obstruction to lifting $Fil_1\subset H_1$ to a Hodge filtration $Fil_2\subset H_2$. It is clear that the obstruction class lies in $H^1(X_1,\Hom(Fil_1,H_1/Fil_1))$. The upshot is that the natural map (compare the map (\ref{eq03}) and Lemma \ref{observs} (1) on its torsor property), obtained from the previous consideration,
\begin{align}\label{obstr}
\varrho: \{X_2\subset X_{3} \} \to  H^1(X_1,\Hom (Fil_1,H_1/Fil_1))
\end{align}
from the set of isomorphism classes of $W_{3}$-liftings of the log curve $X_2$ over $W_2$ to the obstruction space is indeed a \emph{torsor} map under $H^1(X_1,\sT_{X_1/k})$ (see Proposition \ref{anti-obstr}). As a consequence, the ordinary condition on the section $s$ guarantees the existence of a lifting $X_3$ (which actually turns out to be unique) such that one is able to lift $Fil_1$ to a Hodge filtration $Fil_2\subset H_2=C^{-1}_{X_2\subset X_3}(E_2,\theta_2)$. Once this is established, the remaining issues for obtaining a two-periodic Higgs-de Rham flow $(E_2,\theta_2,Fil_2,\phi_2)$ over $X_2/W_2$ which lifts the original one over $X_1/k$ become straightforward. Then we continue the lifting from $W_2$ to $W_3$, and the previous arguments apply. Assuming the existence of a two-periodic Higgs-de Rham flow $(E_{n-1},\theta_{n-1},Fil_{n-1},\phi_{n-1})$ over $X_{n-1}$, we can formulate the generalization of the map (\ref{obstr}) for $n\geq 2$ in the same way:
\begin{align}\label{obstr_gen}
\varrho: \{X_{n}\subset X_{n+1} \} \to  H^1(X_1,\Hom (Fil_1,H_1/Fil_1)).
\end{align}
Before showing the torsor property of $\varrho$, we shall first formulate a sheaf map. Let us write the Higgs field $\theta_1$ of $E_1$ in the following way:
$$
\theta_1: \sT_{X_1/k}\to \End(E_1).
$$
Note that the integrability of the Higgs field means precisely that $\theta_1$ is a morphism of Higgs bundles $(\sT_{X_1/k},0)\to (\End(E_1),\End(\theta_1))$. Applying the inverse Cartier transform to $\theta_1$, we get the corresponding morphism (of flat bundles)
$$
C^{-1}_1(\theta_1):=C^{-1}_{X_1\subset X_2}(\theta_1): F^*\sT_{X_1/k}\to \End(H_1).
$$
Also, there is the obvious projection map
$$
Pr: \End(H_1)\to \Hom(Fil_1,H_1/Fil_1).
$$
So we can consider the composite of the following morphisms of sheaves of abelian groups over $X_1/k$:
$$
\psi: \sT_{X_1/k}\stackrel{F^*}{\longrightarrow}  F^*\sT_{X_1/k}\stackrel{C^{-1}_1(\theta_1)}{\longrightarrow} \End(H_1)\stackrel{Pr}{\longrightarrow} \Hom(Fil_1,H_1/Fil_1),
$$
and its induced map (still denoted by $\psi$ by abuse of notation) on the first cohomology groups:
$$
\psi: H^1(X_1,\sT_{X_1/k})\to H^1(X_1,\Hom(Fil_1,H_1/Fil_1)).
$$
\begin{proposition}\label{anti-obstr}
Use notation as above. Then for $\tau\in \{X_n\subset X_{n+1}\}$ and $\nu\in H^1(X_1,\sT_{X_1/k})$, it holds that
\begin{align}\label{eq501}
\varrho(\tau+\nu)=\varrho(\tau)-\psi(\nu).
\end{align}
\end{proposition}
Let us postpone the proof of this until after showing Theorem \ref{thm04}, which is a direct consequence of the proposition and the next two simple lemmas.
\begin{lemma}\label{commutativity for ordinary}
Suppose the section $s\in \Hom(L_1\otimes \sL,F^*L_1)$ satisfies $A\cap K\cap {\rm Ker}(\phi_s)\neq \emptyset$. Then the following diagram commutes:
\begin{diagram}
 \sT_{X_1/k} &\rTo^{\psi} &\Hom(Fil_1,H_1/Fil_1)\\
     \dTo^{\theta_1}_{\cong}    &              &\dTo_{\cong}\\
           L_1^{-2} &\rTo^{\check{s}^2\circ F^*} &L_1^{-2},\\
 \end{diagram}
where $\check{s}^2: F^*L_1^{-2}\to L_1^{-2}$ is induced by $s$ and the right vertical isomorphism is the obvious one.
\end{lemma}
\begin{proof}
It suffices to show the following diagram commutes:
\begin{diagram}
 F^*\omega_{X_1/k} &\lTo^{\psi^*} &[\Hom(Fil_1,H_1/Fil_1)]^*\\
     \uTo^{F^*\theta^*_1}_{\cong}    &              &\uTo_{\cong}\\
           F^*L_1^{2} &\lTo^{s^2} &L_1^{2}.\\
\end{diagram}
To show this, we first make the following identifications
$$\End(H_1)\cong H_1^{*}\otimes H_1\cong H_1\otimes H_1,$$ where the latter isomorphism is because $H_1\cong H_1^*\otimes \det(H_1)\cong H_1^*$. Under this identification, the dual of the natural projection $Pr$ becomes the inclusion $$Fil_1\otimes Fil_1=L_1^2\otimes \tL^2=L_1^2\hookrightarrow H_1^{\otimes 2}.$$ On the other hand, the composite
$$
H_1\otimes H_1\stackrel{[C^{-1}_{1}(\theta_1)]^*}{\longrightarrow}F^*\omega_{X_1/k}\stackrel{F^*{\theta_1}}{\longrightarrow}
F^*L_1^2
$$
is the self tensor product of the natural projection $H_1\to F^*L_1$. Therefore, the commutativity of the diagram follows from the fact that the Hodge filtration $Fil_1\subset H_1$ is defined by lifting the morphism $s$, that is, the following diagram commutes:
$$
\begin{xy}
   \xymatrix{
      Fil_1\ar[r]^{\hookrightarrow}  & H_1 \ar[r]&F^*L_1  \\
      L_1\otimes \tL  \ar[u]^{\cong}\ar[r]_{=} & L_1\otimes \sL \ar[u]_{\tilde s}  \ar[ur]_{s}&
      }
\end{xy}
$$
\end{proof}
We postpone the proof of the next lemma to the end of this section.
\begin{lemma}\label{inverse cartier transform of torsion bundle}
Let $\tL_{1}=\tL$ be the two torsion line bundle over $X_{1}$ given above. For any sequence of $W_n$-liftings $X_1\subset X_2\subset \cdots\subset X_n\subset \cdots$, let $\tL_n$ be the two torsion line bundle over $X_n$ which is the unique lifting of $\tL_1$. Then one gets inductively that
$$
C_{X_n\subset X_{n+1}}^{-1}(\tL_n,0)=(\tL_{n},\nabla_n),\quad n\geq 2,
$$
where $\nabla_n$ is an integrable connection on $\tL_n$.
\end{lemma}

\begin{proof}[Proof of Theorem \ref{thm04}]
We use induction on $n\geq 2$ and assume we have already lifted the two-periodic Higgs-de Rham flow $(E_1,\theta_1,Fil_1,\phi_1)$ over $X_1/k$ to a two-periodic Higgs-de Rham flow $(E_{n-1},\theta_{n-1},Fil_{n-1},\phi_{n-1})$ over $X_{n-1}/W_{n-1}$ with respect to the $W_n$-lifting $X_{n}$ of the log curve $X_{n-1}$. By Lemma \ref{commutativity for ordinary}, we know the semilinear map $\psi: H^1(X_1,\sT_{X_1/k}\cong L_1^2)\to H^1(X_1,\Hom(Fil_1,H_1/Fil_1)\cong L_1^2)$ is identified with morphism (\ref{ordinary1}) in Definition \ref{ordinary section}. Thus, because of the ordinary assumption on $s$, Proposition \ref{anti-obstr} shows that there exists a unique $W_{n+1}$-lifting $X_{n+1}$ of the log curve $X_n/W_n$ such that the Hodge filtration $Fil_{n-1}$ lifts to a Hodge filtration $Fil_n$ on the flat bundle $C^{-1}_{X_{n}\subset X_{n+1}}(E_{n},\theta_n)$, where $(E_n,\theta_n)$ is the unique lifting of $(E_1,\theta_1)$ to $X_n$. Also, let $\tL_{n}$ be the unique two-torsion line bundle over $X_n$ lifting $\tL_1:=\tL$ over $X_1$.
Note $Gr_{Fil_n}\circ C^{-1}_{X_{n}\subset X_{n+1}}(E_{n},\theta_n)$ and $(E_{n},\theta_n)\otimes (\tL_n,0)$ are two graded Higgs modules over $X_n$ lifting $(E_{1},\theta_{1})\otimes (\tL_{1},0)$ which is also of maximal Higgs field. By the uniqueness in the Lemma \ref{unique lifting of maximal Higgs field} and its proof, there is an isomorphism
$$
\tilde \phi_n: Gr_{Fil_n}\circ C^{-1}_{X_{n}\subset X_{n+1}}(E_{n},\theta_n)\cong (E_{n},\theta_n)\otimes (\tL_n,0)
$$
whose mod $p^{n-2}$ reduction has the form $\mu\phi_{n-1}$ for some $\mu\in W^{\times}_{n-1}$. Take any lift $\tilde \mu\in W^{\times}_n$ of $\mu$ and set $\phi_n=(\tilde\mu)^{-1}\tilde \phi_n$. Then $\phi_n$ lifts $\phi_{n-1}$. Finally, using Lemma \ref{inverse cartier transform of torsion bundle} and the argument in the proof of Corollary \ref{cor of thm 03}, we obtain a two-periodic Higgs-de Rham flow $(E_n,\theta_n,Fil_n,\phi_n)$ over $X_n/W_n$ with respect to $W_{n+1}$-lifting $X_{n+1}$ of $X_n$, which lifts the starting one over $X_{n-1}$. This completes the induction step and therefore the proof of the theorem.
\end{proof}
In the next several paragraphs, we shall digress into a brief account of the construction of the inverse Cartier transform $C^{-1}_{X_n\subset X_{n+1}}(E_n,\theta_n)$ following the theory in \S5 \cite{LSZ}, which may help the reader to understand the local calculations in the proof of Proposition \ref{anti-obstr}. By doing so, we can also lay out some notation used in the proof. First, we give the following:

\begin{definition}\label{log Frobenius lifting over W_n}
Let $X_n, n\geq 2$ be a smooth scheme over $W_n$ with a reduced normal crossing divisor $D_n\subset X_n$ relative to $W_n$ (that is, each component of $D_n$ is flat over $W_n$). A log Frobenius lifting $F_n: X_n\to X_n$ with respect to $D_n$ is a morphism on $X_n$ satisfying (i) its reduction modulo $p$ is the absolute Frobenius on $X_1$; (ii) its restriction to $W_n\subset \sO_{X_n}$ is the Frobenius automorphism of $W_n$; (iii) $F_n^*\sO(-D_n)=\sO(-pD_n)$.
\end{definition}
Given a log Frobenius lifting $F_n$ as above, we get the $\sO_{X_n}$-morphism
$$
dF_n: F_n^*\omega_{X_n/W_n}\to p\omega_{X_n/W_n}\subset \omega_{X_n/W_n},
$$
and hence the Hasse-Witt map over $W_{n-1}$, that is the $\sO_{X_{n-1}}$-morphism
$$
\frac{dF_n}{p}: F_{n-1}^*\omega_{X_{n-1}/W_{n-1}}\to \omega_{X_{n-1}/W_{n-1}}.
$$

Recall that $(E_{n-1},\theta_{n-1},Fil_{n-1}, \phi_{n-1} )$ is a two-periodic Higgs de-Rham flow over $X_{n-1}$ lifting $(E_1,\theta_1,Fil_1,\phi_1)$ over $X_1$. More explicitly,
$$
\phi_{n-1}: Gr_{Fil_{n-1}}\circ C^{-1}_{X_{n-1}\subset X_n}(E_{n-1},\theta_{n-1})\cong (E_{n-1},\theta_{n-1})\otimes (\tL_{n-1},0)
$$
is an isomorphism of graded Higgs bundles over $X_{n-1}$ lifting $\phi_1$,
and it induces the isomorphism $\psi_{n-1}:=\phi_{n-1}\circ Gr_{Fil_{n-1}\otimes Fil_{tr}}\circ C_{X_{n-1}\subset X_n}^{-1}(\phi_{n-1})$
which reads
$$
\psi_{n-1}: Gr_{Fil_{n-1}\otimes Fil_{tr}}\circ C_{X_{n-1}\subset X_n}^{-1}\circ Gr_{Fil_{n-1}}\circ C^{-1}_{X_{n-1}\subset X_n}(E_{n-1},\theta_{n-1})\cong (E_{n-1},\theta_{n-1}),
$$
where $Fil_{n-1}\otimes Fil_{tr}$ is the tensor filtration on $C^{-1}_{X_{n-1}\subset X_n}(E_{n-1},\theta_{n-1})\otimes (\tL_{n-1},\nabla_{n-1})$. \\

{\bf Notation:} Write $(E,\theta)$ for the unique lifting $(E_n,\theta_n)$ of $(E_{n-1},\theta_{n-1})$ to $X_n$. Let $(\bar H,\bar \nabla,\overline{Fil})$ denote the de Rham bundle over $X_{n-1}$:
$$
(\bar H,\bar \nabla)=C_{X_{n-1}\subset X_n}^{-1}(E_{n-1},\theta_{n-1})\otimes (\tL_{n-1},\nabla_{n-1}),\quad \overline{Fil}=Fil_{n-1}\otimes Fil_{tr}.
$$
By the above discussion, $\phi_{n-1}$ induces an isomorphism
$$
\bar \psi: Gr_{\overline{Fil}}(\bar H,\bar \nabla)\cong (\bar E,\bar \theta):=(E_{n-1},\theta_{n-1})
$$
The tuple $(E,\theta,\bar H,\bar \nabla,\overline{Fil},\bar \psi)$ is therefore an object in the category $\sH(X_n)$ (\S5 \cite{LSZ}). Fix an open affine covering $\sU''=\{U''_{i}\}_{i\in I}$ for $X_{n+1}$ such that the $\sO_{U_i''}$-module $\omega_{X_{n+1}/W_{n+1}}(U''_i)$ is free of rank one, which induces the corresponding open affine covering $\sU'=\{U'_i\}$ (resp. $\sU=\{U_i\}$) for $X_n$ (resp. $X_1$) by reduction modulo $p^{n}$ (resp. $p$). Define $U''_{ij}=U''_i\cap U''_j$ and similarly define $U'_{ij}$ and $U_{ij}$. Following a suggestion of the referee, we shall assume in the following that the covering $\sU''$ of $X_{n+1}$ is so chosen that $U''_{ij}\cap D_{n+1}=\emptyset$ for all $i,j$. Certainly, this can be done, since one can take for example $U''_i=X''_{n+1}-\sum_{j\neq i}D_{n+1,i}$ where the boundary divisor $D_{n+1}=\sum D_{n+1,i}$ is written into the summation of disjoint $W_{n+1}$-rational points of $X_{n+1}$. By doing so, we are saved from carrying out extra calculations in the log setting in the proofs of Proposition \ref{anti-obstr} and Lemma \ref{cohclass} and in the verification of the 2-cocycle condition of the Taylor formula. For each $i\in I$, fix a log Frobenius lifting $F''_{i}: U''_i\to U''_i$ whose restriction to $U'_i$ is denoted by $F'_i$. Without loss of generality, we assume that $F''_{i}$ restricts to a morphism on the overlap $U''_{ij}$ for each $i$. Also, we can choose a local coordinate $t''_{ij}$ for $U''_{ij}$ such that $D_{n+1}$ is locally defined by $t''_{ij}$, whose mod $p^{n}$ (resp. mod $p$) reduction will be denoted by $t'_{ij}$ (resp. $t_{ij}$). Set
$$
z_{ij}=\frac{F''^*_i(t''_{ij})-F''^*_j(t''_{ij})}{p}\in \sO_{U'_{ij}}.
$$
The inverse Cartier transform $C_n^{-1}=C^{-1}_{X_n\subset X_{n+1}}$ is the composite of two functors\footnote{The notation differs slightly from \S5 \cite{LSZ}. To simplify the notation, we have not bothered to make the obvious base change of the studied objects over $X'_n=X_n\times_{\sigma}W_n$.}
$$
\sH(X_n)\stackrel{\sT_n}{\to}\widetilde{MIC}(X_n)\stackrel{\sF_n}{\to} MIC(X_n).
$$
The precise definition of each category shall not be recalled here. But since we face a log situation which has not been considered in \cite{LSZ}, we need to point out the necessary modifications in the construction. Roughly speaking, the first functor $\sT_n$ constructs a $\sO_{X_n}$-module with a quasi-nilpotent integrable $p$-connection $(\tilde H, \tilde \nabla)$ from the given tuple $(E,\theta,\bar H,\bar \nabla,\bar \psi)$ in $\sH(X_n)$. The explicit formulas in the second approach to the construction of $\sT_n$ show clearly that it works verbatim in the log situation. The modification occurs in the second functor $\sF_n$. Write $\tilde H_i$ for the restriction $\tilde H|_{U'_{i}}$ and similarly define $\tilde \nabla_i$. Then for each $i$, define a local flat bundle over $U'_i$
$$
H_i:=F^{'*}_i\tilde H_i,\quad \nabla_i: H_i\to H_i\otimes \omega_{U''_{i}/W_n},
$$
where
\begin{align}\label{def of connection}
\nabla_i(f\otimes e):=df\otimes e+f(\frac{dF''_i}{p}\otimes 1)(1\otimes \tilde \nabla_i(e)), \quad f\in \sO_{U'_i}, \ e\in \tilde H_i.
\end{align}
On the overlap $U'_{ij}$, the transition matrix is provided by the Taylor formula. By the assumption that the boundary divisor does not appear in $U'_{ij}$, the usual Taylor formula applies and it reads
$$
G_{ij}(e\otimes 1)=\sum_{n\geq0}\tilde{\nabla}_{\partial_{t'_{ij}}} ^n(e)\otimes \frac{z^n_{ij}}{n!}.
$$
By the well-known 2-cocycle condition of the Taylor formula (also as verified in \S5 \cite{LSZ}), $\{G_{ij}\}$ glues $(H_i,\nabla_i)$ into a global flat bundle $(H,\nabla)$ which is defined to be $C_n^{-1}(E,\theta)$. As a final remark, the functor $\sT_n$ does not depend on the choice of $W_{n+1}$-liftings of $X_n$, but the functor $\sF_n$ does.

Now we can proceed to prove Proposition \ref{anti-obstr}. Let $\hat{X}_{n+1}$ be another $W_{n+1}$-lifting of $X_n$. We put the \emph{hat} on the notations for $X_{n+1}$ to mean the corresponding notations for $\hat X_{n+1}$. For example, $\hat U''_i$ means an open affine subset of $\hat X_{n+1}$. We shall assume that the reductions mod $p^n$ of the corresponding objects for both liftings are the \emph{same}.
\begin{proof}[Proof of Proposition \ref{anti-obstr}]
Use notation as above. Put also
$$
(H',\nabla')=(H,\nabla)\otimes \Z/p^{n-1}, \quad Fil'=Fil_{n-1}.
$$
Caution: The de Rham bundle $(H',\nabla',Fil')$ over $X_{n-1}$ is \emph{not} $(\bar H,\bar \nabla,\overline{Fil})$ (it differs by a twist of $\tL_{n-1}$). For the lifting $X_{n+1}$, let $\{G_{ij}\}$ be the gluing functions for the local flat bundles $\{(H_i,\nabla_i)\}$ to obtain $(H,\nabla)$, provided by the Taylor formula in the functor $\sF_n$. Correspondingly, let $\{\hat G_{ij}\}$ be the gluing functions for $\{(\hat H_i,\hat \nabla_i)\}$ to obtain $(\hat H,\hat \nabla)=C^{-1}_{X_n\subset \hat X_{n+1}}(E,\theta)$ with respect to the lifting $\hat X_{n+1}$. We remind the reader that the bundles with $p$-connection obtained via the functor $\sT_n$ are the same for both liftings so we can identify $\hat H_i$ with $H_i$ for each $i$.

For the fixed lifting $X_{n+1}$, we are going to compute a \v{C}ech-representative of the obstruction class $\varrho(X_n\subset X_{n+1})$ (\ref{obstr_gen}) to lift $Fil'$ to a Hodge filtration $Fil$ (which is a locally direct summand of $H$). We have the following diagram of morphisms:
$$
\begin{xy}
   \xymatrix{
     H_i\ar[dr]_{res}&                  &                        &               &\ar[dl]^{res}H_j\\
                & H_i|_{U''_{ij}}  \ar[rr]^{\stackrel{G_{ij}}{\cong}}&& H_j|_{U''_{ij}}            &
                }
\end{xy}
$$
For each $i\in I$, we choose a local section $e_i$ of $H_i$, satisfying: (i) $e_i\mod p^{n-1}$ is a local basis for $Fil'|_{U'_i}$; (ii) $G_{ij}({\rm res}(e_i))=u'_{ij}{\rm res}(e_j) \mod p^{n-1}$ in $H'|_{U'_{ij}}$ for some unit $u'_{ij}$ (this is possible because $Fil'$ is a global subbundle of $H'$). Therefore, for any lifting $u''_{ij}$ of $u'_{ij}$ to $U''_{ij}$, $G_{ij}({\rm res}(e_i))$ and $u''_{ij}{\rm res}(e_j)$ are two liftings of the same local basis in $Fil'|_{U'_{ij}}$, and $\frac{G_{ij}({\rm res}(e_{i}))-u''_{ij}{\rm res}(e_{j})}{p^{n-1}}$ is then an element in $H_1|_{U_{ij}}$. Recall the projection map $Pr: \End(H_1)\to \Hom(Fil_1,H_1/Fil_1)$ defined above. Regarding $G_{ij}$ as an element in $\End(H|_{U''_{ij}})$, $O_{ij}:=Pr\circ \frac{G_{ij}-u''_{ij}\cdot Id}{p^{n-1}}$ defines a section of $\Hom(Fil_1,H_1/Fil_1)$ over $U_{ij}$. Now we proceed to verify the 2-cocycle condition for $\{O_{ij}\}$. Complete the above local section $e_i$ into a local basis $\{e_i,f_i\}$ of $H_i$ and then represent the transition $G_{ij}$ by the matrix
$\left(
  \begin{array}{cc}
    a_{ij} & b_{ij} \\
    c_{ij} & d_{ij} \\
  \end{array}
\right)$.
The assumption (ii) on $e_i$ implies $c_{ij}=0\mod p^{n-1}$. Let $\bar c_{ij}=\frac{c_{ij}}{p^{n-1}}\in \sO_{U_{ij}}$. For a temporary use, let us put \emph{bar} over a local section to mean its reduction mod $p$, and let $\pi: H_1\to H_1/Fil_1$ be the projection. Then $O_{ij}(\bar e_i|_{U_{ij}})=\bar c_{ij}\pi(\bar f_j|_{U_{ij}})$. One derives the 2-cocyle condition from the 2-cocycle condition satisfied by $\{G_{ij}\}$ as follows: over $U''_{ijk}:=U''_{i}\cap U''_{j}\cap U''_{k}$, one has the equality
$$
\left(
  \begin{array}{cc}
    a_{jk} & b_{jk} \\
    c_{jk} & d_{jk} \\
  \end{array}
\right)\left(
  \begin{array}{cc}
    a_{ij} & b_{ij} \\
    c_{ij} & d_{ij} \\
  \end{array}
\right)=\left(
  \begin{array}{cc}
    a_{ik} & b_{ik} \\
    c_{ik} & d_{ik} \\
  \end{array}
\right),
$$
which gives particularly the following equality in $\sO_{U''_{ijk}}$:
$$a_{ij}c_{jk}+c_{ij}d_{jk}=c_{ik}.$$ Dividing both sides by $p^{n-1}$, one gets $\bar a_{ij}\bar c_{jk}+\bar c_{ij}\bar d_{jk}=\bar c_{ik}$, where $\bar a_{ij}$ (resp. $\bar d_{jk}$) is the mod $p$ reduction of $a_{ij}$ (resp. $d_{jk}$). On the other hand, over $U_{ijk}$,
$$
(O_{ij}+O_{jk})(\bar e_i|_{U_{ijk}})=(\bar c_{ij}\bar d_{jk}+\bar a_{ij}\bar c_{jk})\pi(\bar f_k|_{ijk}),
$$
and $O_{ik}(\bar e_i|_{U_{ijk}})=\bar c_{ik}\pi(\bar f_k|_{U_{ijk}})$. Thus $O_{ij}+O_{jk}=O_{ik}$ holds. Moreover, it is easy to check further that the 1-cocyle $\{O_{ij}\}$ is a coboundary $dh, h=(h_{i})$ iff $c_{ij}$ can be modified to be zero (so that $\{e_i\}$s glue into a sub line bundle $Fil$ of $H$ lifting $Fil'$. So $\{O_{ij}\}$ gives us a representative for the obstruction class lifting the Hodge filtration.

Now we shall measure the difference of two obstruction classes coming from two liftings $\hat X_{n+1}$ and $X_{n+1}$. Recall the function $\hat z_{ij}=\frac{\hat{F}''^*_i(t''_{ij})-\hat{F}''^*_j(t''_{ij})}{p}\in \sO_{U'_{ij}}$ in the Taylor formula $\hat G_{ij}$. Define $\alpha_{ij}=\hat z_{ij}-z_{ij}$. By the assumption on the Frobenius liftings, $\alpha_{ij}$ is divisible by $p^{n-1}$. A direct calculation shows the following relation between the transition matrices $\hat G_{ij}$ and $G_{ij}$:
$$
\hat G_{ij}=W_{ij}\cdot G_{ij},
$$
where for $e\in \tilde H|_{U'_{ij}}$,
$$
W_{ij}(e\otimes 1)=e\otimes1 + \tilde{\nabla}_{\partial t_{ij}'}(e)\otimes \alpha_{ij}.
$$
Thus, the difference between the obstruction classes can be represented by
$$
\{Pr\circ(\frac{\hat{G}_{ij}-u''_{ij}\cdot Id}{p^{n-1}})-Pr\circ(\frac{G_{ij}-u''_{ij}\cdot Id}{p^{n-1}})=
Pr\circ(\frac{W_{ij}-Id}{p^{n-1}})\circ \bar{G}_{ij}\},
$$
where $\bar{G}_{ij}$ denotes the reduction of $G_{ij}$ mod $p$.  As the reduction of $p$-connection $\tilde \nabla$ mod $p$ is just the Higgs field $\theta_1$, it follows that
$$
\frac{W_{ij}-Id}{p^{n-1}}= F^*(\theta_{1,\partial_{ t_{ij}}})\frac{\alpha_{ij}}{p^{n-1}}.
$$
Let $\nu=[\hat X_{n+1}]-[X_{n+1}]\in H^1(X_1,\sT_{X_1/k})$. The next lemma shows that the cycle $\{-\frac{\alpha_{ij}}{p^{n-1}}(\partial_{t_{ij}}\otimes 1)\}$ represents the class $F^*\nu$ in $H^1(X_1,F^*\sT_{X_1/k})$. Thus we get $C_1^{-1}(\theta_1)(F^*\nu)=-C_1^{-1}(\theta_{1,\partial_{t_{ij}}})\frac{\alpha_{ij}}{p^{n-1}}$ over the overlap $U_{ij}$, and now we need a clearer expression of it. For this, we choose a local basis $e_i=\{e_{i,1},e_{i,2}\}$ of $E_1|_{U_i}$ and $e_j=\{e_{j,1},e_{j,2}\}$ of $E_1|_{U_j}$, and let $M_{ij}$ be the transition matrix between the two bases on the overlap $U_{ij}$, that is, $e_j|_{U_{ij}}=e_i|_{U_{ij}}M_{ij}$. Then $e_i\otimes 1$ (resp. $e_j\otimes 1$) gives a local basis for $F^*E_1|_{U_i}$ (resp. $F^*E_1|_{U_j}$) and the corresponding transition matrix (to glue into $F^*E_1$) is $F^*M_{ij}$. By the construction of the inverse Cartier transform, $H_1$ is obtained by gluing $F^*E_1|_{U_i}$ and $F^*E_1|_{U_j}$ (using the above basis) via the gluing matrix $\bar G_{ij}\cdot(F^*M_{ij})$ (the so-called exponential twisting in \cite{LSZ0}), where $\bar G_{ij}=Id+F^*(\theta_{1,\partial_{t_{ij}}})\bar z_{ij}$ with $\bar z_{ij}$ the mod $p$ reduction of $z_{ij}$. Thus, on the overlap, the morphism
$$
F^*E_1|_{U_{ij}}\stackrel{F^*(\theta_{1,\partial t_{ij}})}{\longrightarrow} F^*E_1|_{U_{ij}}
$$
is expressed into $F^*M_{ij}F^*\theta_{1,\partial t_i}$, where we use the restriction of $e_i\otimes 1$ on the source and the restriction of $e_j\otimes 1$ on the target as local basis. On the other hand, the morphism
$$
H_1|_{U_{ij}}\stackrel{C^{-1}_1(\theta_{1,\partial t_{ij}})}{\longrightarrow} H_1|_{U_{ij}}
$$
can be expressed as $\bar G_{ij}F^*M_{ij}F^*\theta_{1,\partial t_i}$ using the same bases as above. Thus, we conclude that
$$
C_1^{-1}(\theta_{1,\partial t_{ij}})=\bar G_{ij}\circ F^*(\theta_{1,\partial_{ t_{ij}}})=F^*(\theta_{1,\partial_{ t_{ij}}})\circ \bar G_{ij},
$$
where the second equality can be seen by the expression of $\bar G_{ij}$. Therefore $\{ Pr\circ\frac{W_{ij}-Id}{p^{n-1}}\circ \bar{G}_{ij}\} $ represents the class
$-Pr\circ C_1^{-1}(\theta_1)\circ F^*(\nu)$. This finishes the proof.
\end{proof}
The proof of the following lemma is analogous to that of Lemma \ref{observs} (1).
\begin{lemma}\label{cohclass}
Use notation as above. Then the image of $\nu=[\hat X_{n+1}]-[X_{n+1}]$ under the map $F^*: H^1(X_1,\mathcal{T}_{X_1/k})\to H^1(X_1, F^*\mathcal{T}_{X_1/k})$ is represented by the class $\{-\frac{\alpha_{ij}}{p^{n-1}}(\partial_{t_{ij}}\otimes 1)\}$.
\end{lemma}
\begin{proof}
Each open affine $U'_i$ (resp. $U'_{ij}$ which is again affine) has a unique $W_{n+1}$-lifting $V''_i$ (resp. $V''_{ij}$) up to isomorphism (we shall suppress the choice of isomorphisms in our argument). Fix embeddings $\{V''_{ij}\to V''_j\}$, and corresponding to the $W_{n+1}$-lifting $\hat X_{n+1}$ (resp. $X_{n+1}$) of $X_n$, there are embeddings $\{\hat{g}_{ij}: V_{ij}''\to
V_i''\}$ (resp. $\{g_{ij}:V_{ij}''\to V_i''\}$).
Note that $\hat g_{ij}$ and $g_{ij}$ have the same
reduction modulo $p^{n}$. Thus $\frac{\hat g_{ij}^*-g_{ij}^*}{p^{n}}: \mathcal{O}_{U_{i}}\to
\mathcal{O}_{U_{ij}}$ is a $k$-derivation and gives rise to a \v{C}ech-representative $\{\nu_{ij}\}$
for $\nu\in H^1(\sT_{X_1/k})$. Now we choose and then fix (log) Frobenius liftings $\{G''_i:V_i''\to V_i''\}_{i\in I}$ which are assumed to induce Frobenius liftings over $V''_{ij}$. Thus we get two composite morphisms
$$
g_{ij}\circ G''_j,\ G''_i\circ g_{ij}: V''_{ij}\to V''_i
$$
which are two (log) Frobenius liftings over $V''_{ij}$. By the identification $U''_{ij}=V''_{ij}$, we can assume $F''_i=G''_i\circ g_{ij}$ and $F''_j=g_{ij}\circ G''_j$ on the overlap. Then we get
$$
z_{ij}=\frac{g_{ij}^*\circ (G''_i)^*(t''_{ij})-(G''_j)^*\circ g_{ij}^*(t''_{ij})}{p}.
$$
Similarly, we obtain
$$
\hat{z}_{ij}=\frac{\hat{g}_{ij}^*\circ
(G''_i)^*(t''_{ij})-(G''_j)^*\circ \hat{g}_{ij}^*(t''_{ij})}{p}.
$$
Because of the mod $p^n$-reduction property of $\hat g_{ij}$ and $g_{ij}$, the difference $\alpha_{ij}=\hat z_{ij}-z_{ij}$ is divisible by $p^{n-1}$, and we get
$$
\frac{\alpha_{ij}}{p^{n-1}}=\frac{\hat z_{ij}-z_{ij}}{p^{n-1}}=\frac{1}{p^{n}}\left((\hat g_{ij}^*-g_{ij}^*)\circ
(G''_i)^*(t''_{ij})-(G''_j)^*\circ
(\hat g_{ij}^*(t''_{ij})-g_{ij}^*(t''_{ij}))\right).
$$
As $\frac{\hat g_{ij}^*-g_{ij}^*}{p^{n}}$ is a derivation and
$d(G''_i)^*(t''_{ij})$ is divisible by $p$,
$\frac{(\hat g_{ij}^*-g_{ij}^*)\circ (G''_i)^*(t''_{ij})}{p^{n}}=0$.
Therefore,
$$
-\frac{\alpha_{ij}}{p^{n-1}}=\frac{(G''_j)^*\circ (\hat g_{ij}^*(t''_{ij})- g_{ij}^*(t''_{ij}))}{p^n}=F^*\circ(\frac{
\hat g_{ij}^*-g_{ij}^*}{p^{n}})(t_{ij}),
$$
which means $\{-\frac{\alpha_{ij}}{p^{n-1}}(\partial t_{ij}\otimes 1)\}$ represents the class
$\{F^*\nu\}\in H^1(X_1,F^*\sT_{X_1/k})$.
\end{proof}

Finally, we shall provide the proof of Lemma \ref{inverse cartier transform of torsion bundle}, which follows directly from the construction of the inverse Cartier transform.
\begin{proof}[Proof of Lemma \ref{inverse cartier transform of torsion bundle}]
Start with $n=1$. Note that $(\tL_{1},0,Fil_{tr},id)$ defines a one-periodic Higgs de-Rham flow over $X_1$. Therefore, the tuple $$(\tL_2,0,\tL_1,\nabla_{can},Fil_{tr},id)$$ is an object in $\sH(X_2)$, where $\nabla_{can}$ is the canonical connection on $F^*\tL_1=\tL_1^p=\tL_1$. Let $(\tilde H_2,\tilde \nabla_2)$ be the corresponding bundle with $p$-connection. According to the second approach in the construction of the functor $\sT_2$ given in \cite{LSZ}, $\tilde H_2$ can be identified with the cokernel of the morphism
$$
\tL_2\to \tL_1\oplus \tL_2, \quad x\mapsto (\bar x, -px),
$$
where $\bar x$ is the reduction of $x$ mod $p$. It factors as
$$
\tL_2\stackrel{\mod p}{\twoheadrightarrow} \tL_1\stackrel{(id,-p)}{\hookrightarrow} \tL_1\oplus \tL_2,
$$
whose image is clearly identified with the kernel of the morphism
$$
\tL_1\oplus \tL_2\stackrel{\left(
                 \begin{array}{c}
                   p \\
                   id \\
                 \end{array}
               \right)
}{\twoheadrightarrow} \tL_2.
$$
Hence $\tilde H_2=\tL_2$. Let $e$ be a local section of $\tL_2$ such that its mod $p$ reduction is a local ($\nabla_{can}$-)flat basis of $\tL_1$. Then $e$ is flat with respect to the $p$-connection $\tilde \nabla_2$ by its definition. Next, by Formula (\ref{def of connection}), it is also clear that $1\otimes e$ is a flat basis of the local flat bundle $(H_2,\nabla_2)=\sF_{2}(\tilde H_2,\tilde \nabla_2)$ restricted to some open affine subset. Suppose the transition function of $\tilde H_2=\tL_2$ is given by $g_{ij}$. Then by the Taylor formula expressed via the local flat bases, we know that the transition function for $H_2$ is given by  $g^p_{ij}$ which simply means $H_2=\tL_2^p=\tL_2$. Then one constructs inductively the one-periodic Higgs-de Rham flow $(\tL_n,0, Fil_{tr},id)$ over $X_{n}/W_n$ (or equivalently the object $(\tL_n,0,\tL_{n-1},\nabla_{n-1},Fil_{tr},id)$ in $\sH(X_n)$) and shows, by local calculations similar to the above, that $C^{-1}_{X_n\subset X_{n+1}}(\tL_n,0)=(\tL_{n},\nabla_{n})$ for $n\geq 2$ and $\tL_n$ admits a local flat basis with respect to $\nabla_n$.
\end{proof}

\section{Appendix: Cartier transform and inverse Cartier transform in a special log case}
Let $k$ be a perfect field of positive characteristic and $X$ a smooth algebraic variety over $k$. Let $D=\sum_iD_i\subset X$ be a simple normal crossing divisor, where $D_i$ is a smooth irreducible component. This gives rise to one of standard examples of the logarithmic structure on $X$ (Example 1.5 (1) \cite{KKa}). The aim of this appendix is to provide the Cartier/inverse Cartier transform of Ogus-Vologodsky \cite{OV} in this special log case. Note the thesis of D. Schepler (see \cite{S}) has treated this issue in a more general log setting. In particular, we believe our main result Theorem \ref{equivalence} is really contained in Corollary 4.11 (iii) in \cite{S}. Thus, the main point of this appendix is to provide the exponential twisting approach (see \cite{LSZ0} for $D=\emptyset$) to the Cartier/inverse Cartier transform of Ogus-Vologodsky and Schepler in this special log setting. Also, the construction of the inverse Cartier transform serves as the basis for its lifting to a truncated Witt ring as sketched in \S 5. We acknowledge our referee for pointing out an error in the original formulation of the category $\MIC^0_{\leq n}(X_{\log}/k)$.

Let $\omega_{X_{\log}/k}=\Omega_X(\log D)$ be the sheaf of log differential forms which is a locally free $\sO_X$-module. Let $\sT_{X_{\log}/k}$ be its $\sO_X$-dual. Let us recall the short exact sequence of the residue map:
$$
0\to \omega_{X/k}\to \omega_{X_{\log}/k}\stackrel{\oplus res_{D_i}}{\longrightarrow} \oplus_i\sO_{D_i}\to 0,
$$
where $\omega_{X/k}$ is the sheaf of regular differential one forms on $X$.  First we formulate two categories: for a nonnegative integer $n$, let $\HIG_{\leq n}(X_{\log}/k)$ be the category of nilpotent logarithmic Higgs sheaves over $X/k$ of exponent $\leq n$. Any object $(E,\theta)\in \HIG_{\leq n}(X_{\log}/k)$ is of form
$$
\theta: E\to E\otimes \omega_{X_{\log}/k}
$$
satisfying $$
\theta_{\partial_1}\cdots\theta_{\partial_{n+1}}=0
$$
for any local sections $\partial_1,\cdots, \partial_{n+1}$ of $\sT_{X_{\log}/k}$\footnote{The convention of exponent adopted here differs from the one used in \cite{LSZ0} which originated from N. Katz \cite{KA}, but conforms with the one used in Ogus-Vologodsky \cite{OV}.}. Let $\HIG^{0}_{\leq n}(X_{\log}/k)$ be the full subcategory of $\HIG_{\leq n}(X_{\log}/k)$ with the tor condition :
 \begin{align}\label{tor for higgs}
 \mathcal{T}or_1(E,F_*\sO_{D_i})= 0, \quad \forall \ i.
 \end{align}
Note that the tor condition is a local condition: $\mathcal{T}or_1(E,F_*\sO_{D_i})= 0$ iff for every closed point $x\in D_i$, $Tor^{\sO_{X,x}}_1(E_x,F_*\sO_{D_{i.x}})= 0$. Note also that the tor condition is obviously satisfied in the case that either
$E$ is locally free or $D$ is simply an empty set. Correspondingly, let $\MIC_{\leq n}(X_{\log}/k)$ be the category of nilpotent logarithmic flat sheaves of exponent $\leq n$ (its object is similarly defined with the nilpotent condition referring to its $p$-curvature) and $\MIC^{0}_{\leq n}(X_{\log}/k)$ the full subcategory of $\MIC_{\leq n}(X_{\log}/k)$ consisting of objects $(H,\nabla)$ satisfying the tor condition
 \begin{align}\label{tor for flat}
 \mathcal{T}or_1(H,\sO_{D_i})= 0, \quad \forall \ i,
 \end{align}
 together with the residue condition that the \emph{residue} of $\nabla$ is nilpotent of exponent $\leq n$ which is explained as follows: recall that for a logarithmic flat (resp. Higgs) sheaf $(H,\nabla)$ (resp. $(E,\theta)$), the residue along $D_i$ is an $\sO_{D_i}$-linear morphism $$\textrm{Res}_{D_i}\nabla: H|_{D_i}:=H\otimes \sO_{D_i} \to H|_{D_i}$$ (resp. $\textrm{Res}_{D_i} \theta: E|_{D_i}\to E|_{D_i}$) obtained from the composite
$$
H\stackrel{\nabla}{\longrightarrow} H\otimes \omega_{X_{\log}/k}\stackrel{id\otimes \textrm{res}_{D_i}}{\longrightarrow} H\otimes \sO_{D_i}.
$$
If for each $i$, $(\textrm{Res}_{D_i} \nabla)^{n+1}=0$ (resp. $(\textrm{Res}_{D_i} \theta)^{n+1}=0$), we call the residue of $\nabla$ (resp. $\theta$) nilpotent of exponent $\leq n$. Note that, for $(E,\theta)\in \HIG_{\leq n}(X_{\log}/k)$, $(\textrm{Res}\ \theta)^{n+1}=0$ holds automatically; but this is not the case for $\MIC_{\leq n}(X_{\log}/k)$.
\begin{theorem}[Corollary 4.11 (iii) \cite{S}; Theorem 1.2 \cite{LSZ0} for $D=\emptyset$]\label{equivalence}
Assume $(X,D)$ is $W_2$-liftable. Then there is an equivalence of categories
$$
\HIG^0_{\leq p-1}(X_{\log}/k)\xrightleftharpoons[\ C\ ]{\ C^{-1}\ }\MIC^0_{\leq p-1}(X_{\log}/k).
$$
\end{theorem}
Our argument follows the line of \cite{LSZ0} which is completely elementary. The tor condition first appears in the logarithmic analogue of the Cartier descent theorem due to A. Ogus (see Theorem 1.3.4 \cite{O}), from which we borrow for the following simple reason:
\begin{lemma}\label{cartier descent}
Let $\HIG^0_{\leq 0}(X/k)$ (resp. $\MIC^0_{\leq 0}(X/k)$) be the full subcategory of $\HIG_{\leq 0}(X/k)$ (resp. $\MIC_{\leq 0}(X/k)$) with the tor condition (\ref{tor for higgs}) (resp. \ref{tor for flat}). Then the category $\HIG^0_{\leq 0}(X_{\log}/k)$ is equal to $\HIG^0_{\leq 0}(X/k)$ while the category $\MIC^0_{\leq 0}(X_{\log}/k)$ is equal to $\MIC^0_{\leq 0}(X/k)$. The classical Cartier descent theorem gives an equivalence of categories between $\HIG^0_{\leq 0}(X/k)$ and $\MIC^0_{\leq 0}(X/k)$.
\end{lemma}
\begin{proof}
Since $\omega_{X/k}$ is naturally a subsheaf of $\omega_{X_{\log}/k}$, $\HIG^0_{\leq 0}(X/k)$ (resp. $\MIC^0_{\leq 0}(X/k)$) is a subcategory of $\HIG^0_{\leq 0}(X_{\log}/k)$ (resp. $\MIC^0_{\leq 0}(X_{\log}/k)$). Take an object $(H,\nabla)$ in $\MIC^0_{\leq 0}(X_{\log}/k)$. By the tor condition, one has a short exact sequence
$$
0\to H\otimes \omega_{X/k}\to H\otimes \omega_{X_{\log}/k}\stackrel{id\otimes \textrm{res}}\to \oplus_i H\otimes \sO_{D_i}\to 0.
$$
The residue condition says that $\nabla(H)\subset H\otimes \omega_{X_{\log}/k}$ goes to zero under the morphism $id\otimes \textrm{res}$ and therefore is contained in $H\otimes \omega_{X/k}$ by the exactness of the previous sequence. Thus the logarithmic flat sheaf $(H,\nabla)$ is indeed a flat sheaf on $X$. So
$$
\MIC^0_{\leq 0}(X/k)=\MIC^0_{\leq 0}(X_{\log}/k).
$$
The equality $\HIG^0_{\leq 0}(X/k)=\HIG^0_{\leq 0}(X_{\log}/k)$ is clear. It remains to show the Cartier descent preserves the tor conditions. Let $(E,0)$ in $\HIG_{\leq 0}(X/k)$, and let $(H=F^*E, \nabla_{can})\in \MIC_{\leq 0}(X/k)$ be the corresponding object. Since $X/k$ is smooth, the absolute Frobenius morphism $F$ is flat. By the flat base change for tor, one has the equality of $\sO_{X}$-modules:
$$
F_*\mathcal{T}or_1(H,\sO_{D_i})= \mathcal{T}or_1(E,F_*\sO_{D_i}).
$$
From this, one sees that $(E,0)$ satisfies the tor condition (\ref{tor for higgs}) iff $(H,\nabla)$ satisfies the tor condition (\ref{tor for flat}). This completes the proof.
\end{proof}
We can proceed to the proof of Theorem \ref{equivalence}, whose idea is to reduce the general case to the classical Cartier descent (Lemma \ref{cartier descent}) via an exponential twisting \cite{LSZ0}.
\begin{proof}
Given the explicit exposition of the constructions in \cite{LSZ0} for the case where $D$ is absent, we shall not repeat the whole argument but rather emphasize the new ingredients in the new situation. Fix a $W_2$-lifting $(\tilde X,\tilde D)$ of $(X,D)$, see Definition 8.11 \cite{EV}. Recall that, for an open affine subset $\tilde U\subset \tilde X$, a \emph{log Frobenius lifting} over $\tilde U$, respecting the divisor $\tilde D_U:= \tilde D\cap \tilde U$, is a morphism over the Frobenius on $W_2$
$$
\tilde F_{(\tilde U,\tilde D_U)}: \sO_{\tilde U}\to \sO_{\tilde U}
$$
lifting the absolute Frobenius morphism on $U:=\tilde U\times k$ and satisfying
$$
\tilde F_{(\tilde U,\tilde D_U)}^*\sO_{\tilde U}(-\tilde D_U)=\sO_{\tilde U}(-p\tilde D_U).
$$
Such a lifting exists and two such liftings differ by an element in $F^*\sT_{X_{\log}/k}$ over $U$. See Propositions 9.7 and 9.9 in \cite{EV}. Then, proceeding in the same manner as in Section 2.2 \cite{LSZ0} for the inverse Cartier transform $C^{-1}$, one obtains $(H,\nabla)=C^{-1}(E,\theta)$, which clearly belongs to $\MIC_{\leq p-1}(X_{\log}/k)$. Since $E$ satisfies the tor condition and $H$ is locally $F^*E$, it follows that $H$ also satisfies the tor condition from the proof of Lemma \ref{cartier descent}.  Now we use $\zeta$ to represent the map
$$
\frac{d\tilde F_{(\tilde U,\tilde D_U)}}{p}: F^*\omega_{X_{\log}/k}(U)\to \omega_{X_{\log}/k}(U).
$$
Then locally over $U$, the connection
is given by
$$
\nabla|_{U}=\nabla_{can}+\zeta(F^*\theta|_{U}).
$$
It follows that $\Res_{D_i}\nabla|_{U}=F^*\Res_{D_i}\theta|_{U}$ and therefore $(H,\nabla)\in \MIC^0_{\leq p-1}(X_{\log}/k)$. Conversely, given an object $(H,\nabla)\in \MIC^0_{\leq p-1}(X_{\log}/k)$, one proceeds as Section 2.3 \cite{LSZ0}. In the same way, one shows that the new connection
$$
\nabla'|_{U}=\nabla|_{U}+\zeta(\psi|_{U}),
$$
where $\psi=\psi_{\nabla}: H\to H\otimes F^*\omega_{X_{\log}/k}$ is the $p$-curvature of $\nabla$, has vanishing $p$-curvature. However, there is one new ingredient here, namely, the following.
\begin{claim}\label{new equality}
The following diagram commutes:
\begin{align}\label{comm diagram}
\xymatrix{
H \ar[d]^{} \ar[r]^-{\psi_{\nabla} } &H\otimes F^*\omega_{X_{\log}/k} \ar[r]^-{id\otimes \zeta }& H\otimes \omega_{X_{\log}/k} \ar[d]^{id\otimes \textrm{res}_{D_i}} \\
H\otimes \sO_{D_i} \ar[rr]^{-\Res_{D_i}\nabla} &&H\otimes \sO_{D_i}. }
\end{align}
\end{claim}
Granted the truth of the claim (which is proven below), the residue of $\nabla'|_{U}$ vanishes. In this way, we reduce it to Lemma \ref{cartier descent}.
\end{proof}
Here is the proof of the claim just mentioned above:
\begin{proof}
Let $x\in D_i$ be a closed point and let $\{t_1,\cdots,t_d\}$ be a set of local coordinates for an open affine neighborhood $\tilde U\subset \tilde X$ of $x$, such that $\tilde D_{\tilde U}$ is defined by $\prod_{1\leq j\leq r}t_j=0$. Then $\{\frac{dt_j}{t_j}, dt_l\}_{1\leq j\leq r, r+1\leq l\leq d}$ is a basis for $\Gamma(\tilde U,\omega_{X_2/W_2})$, where $\omega_{X_2/W_2}$ stands for the sheaf $\Omega_{\tilde X}(\log \tilde D)$. Assume moreover that $\tilde{D}_{i,\tilde U}$ is defined by $t_1=0$. Let $R=\Gamma(\tilde U,\sO_{\tilde U})$. Then a log Frobenius lifting over $\tilde U$ is of the following form:
$$
\tilde F(t_j)=t_j^pu_j, \ u_j=1+pa_j, \ 1\leq j\leq r, \quad \tilde F(t_l)=t_l^p+pb_l,  \ r+1\leq l\leq d, \quad  a_j, b_l\in R.
$$
Consider the composite over $U$ (here we write $\zeta_{\tilde F}$ for $\zeta$ to distinguish the choice of $\tilde F$ in the definition):
$$
F^*\omega_{X_{\log}/k}\stackrel{\zeta_{\tilde F}}{\to}\omega_{X_{\log}/k}\stackrel{\textrm{res}_{D_i}}{\longrightarrow}\sO_{D_i}.
$$
We claim that it is independent of the choice of $\tilde F$. Indeed, because
$$
\zeta_{\tilde F}(\frac{dt_j}{t_j})=\frac{dt_j}{t_j}+da_j; \quad \zeta_{\tilde F}(dt_l)=t_l^{p-1}dt_l+db_l,
$$
it follows that $\textrm{res}_{D_i}\circ \zeta_{\tilde F}(\frac{dt_1}{t_1})=1$ and
$$
\textrm{res}_{D_i}\circ \zeta_{\tilde F}(\frac{dt_j}{t_j})=\textrm{res}_{D_i}\circ \zeta_{\tilde F}(dt_l)=0, \quad 2\leq j\leq r, r+1\leq l\leq d.
$$
Thus, we can take the standard log Frobenius lifting $\tilde F(t_i)=t_i^p$ for each $i$ in the following calculation. Using the standard Frobenius lifting, the composite $(id\otimes \zeta)\circ \psi$ is expressed into
$$
\psi=\psi_1\otimes \frac{dt_1}{t_1}+\cdots+\psi_r\otimes \frac{dt_r}{t_r}+t_{r+1}^{p-1}\psi_{r+1}\otimes dt_l+\cdots+t_{d}^{p-1}\psi_{d}\otimes dt_d,
$$
where $\psi_j$ (resp. $\psi_l$) is the endomorphism of $H$ over $U$ defined by $\psi_{1\otimes t_j\partial_{t_j}}$ (resp. $\psi_{1\otimes \partial_{t_l}}$). Similarly, we express the connection over $U$ into
$$
\nabla=\nabla_1\otimes\frac{dt_1}{t_1}+\cdots+\nabla_r\otimes \frac{dt_r}{t_r}+\nabla_{r+1}\otimes dt_l+\cdots+\nabla_{d}\otimes dt_d.
$$
Since we are considering the residue along $D_i$ which is defined by $t_1=0$, to show the commutativity of the diagram (\ref{comm diagram}), it suffices to verify the equality
$$
\psi_1|_{t_1=0}=-\nabla_1|_{t_1=0},
$$
whose verification is formally like the curve case. For simplicity, we write $t$ for $t_1$. By an elementary calculation, one finds $(t\partial_{t})^p=t\partial_{t}$. Thus,
$$
\psi_1=\psi_{1\otimes t\partial_t}=\nabla^p_{t\partial_t}-\nabla_{t\partial_t},
$$
so we conclude that its residue $\psi_1|_{t=0}$ at $\{t=0\}$ equals
$$
\nabla^p_{t\partial_t}|_{t=0}-\nabla_{t\partial_t}|_{t=0}=
(\nabla_{t\partial_t}|_{t=0})^p-\nabla_{t\partial_t}|_{t=0}=-\nabla_{t\partial_t}|_{t=0}
$$
which is just $-\nabla_1|_{t=0}$. Note the second equality in above follows from the assumption $(\Res_{D_i}\nabla)^p=0$. The claim follows.
\end{proof}

{\bf Acknowledgement:} We would like to thank Jie Xia for sending us the manuscript \cite{Xia} and drawing our attention to the work of S. Mochizuki. We thank warmly Adrian Langer for helpful comments on Remark \ref{general inject}. We are grateful to the referee for pointing out several errors in the original manuscript, which include an error in the early construction of the algebraic structure on the periodic cone $K$ and an error in the original formulation of the Ogus-Vologodsky correspondence in the special log case (Theorem \ref{equivalence}), and kindly providing us many valuable comments and suggestions, especially in log geometry. The improvement in the exposition of the work is due to the great efforts made by the referee. We thank him/her sincerely.

\end{document}